\documentclass[12pt,a4paper,reqno]{amsart} 

\usepackage[T1]{fontenc}       

\usepackage{amsfonts}          

\usepackage{amsmath}

\usepackage{amssymb}

\usepackage{overpic}
\usepackage{euscript} 
\usepackage{eurosym}
\usepackage{comment}
\usepackage{mathrsfs} 
\usepackage{ifthen}
\usepackage{esint} 
\usepackage{color}

\long\def\symbolfootnote[#1]#2{\begingroup%
\def\thefootnote{\fnsymbol{footnote}}\footnote[#1]{#2}\endgroup}

{\qed\vspace{5pt}}

\usepackage{amsthm}

\newtheoremstyle{lause}
{5pt}
{5pt}
{\slshape}
{\parindent}
{\bfseries}
{.}
{.5em}
{}

\theoremstyle{lause}

\newtheoremstyle{maaritelma}
{5pt}
{5pt}
{\rmfamily}
{\parindent}
{\bfseries}
{.}
{.5em}
{}

\theoremstyle{maaritelma}
\newtheoremstyle{lause}
{5pt}
{5pt}
{\slshape}
{\parindent}
{\bfseries}
{.}
{.5em}
{}

\theoremstyle{lause}
\newtheorem{theorem}{Theorem}[section]
\newtheorem{lemma}[theorem]{Lemma}

\newtheorem{corollary}[theorem]{Corollary}
\newtheorem{problem}[theorem]{Problem}

\newtheoremstyle{maaritelma}
{5pt}
{5pt}
{\rmfamily}
{\parindent}
{\bfseries}
{.}
{.5em}
{}

\theoremstyle{maaritelma}
\newtheorem{definition}[theorem]{Definition}

\newtheorem{example}[theorem]{Example}
\newtheorem{remark}[theorem]{Remark}

\numberwithin{equation}{section}

\begin{document}

\thispagestyle{empty}

\begin{center}

{\large{\textbf{Minimum Riesz energy problems with external fields}}}

\vspace{18pt}

\textbf{Natalia Zorii}

\vspace{18pt}

\footnotesize{\address{Institute of Mathematics, Academy of Sciences
of Ukraine, Tereshchenkivska~3, 01601,
Kyiv-4, Ukraine\\
natalia.zorii@gmail.com }}

\end{center}

\vspace{12pt}

{\footnotesize{\textbf{Abstract.} The paper deals with minimum energy problems in the presence of external fields with respect to the Riesz kernels $|x-y|^{\alpha-n}$, $0<\alpha<n$, on $\mathbb R^n$, $n\geqslant2$.
For quite a general (not necessarily lower semicontinuous) external field $f$, we obtain necessary and/or sufficient conditions for the existence of $\lambda_{A,f}$ minimizing the Gauss functional
\[\int|x-y|^{\alpha-n}\,d(\mu\otimes\mu)(x,y)+2\int f\,d\mu\]
over all positive Radon measures $\mu$ with $\mu(\mathbb R^n)=1$, concentrated on quite a general (not necessarily closed) $A\subset\mathbb R^n$. We also provide various alternative characterizations of the minimizer $\lambda_{A,f}$, analyze the continuity of both $\lambda_{A,f}$ and the modified Robin constant for monotone families of sets, and give a description of the support of $\lambda_{A,f}$. The significant improvement of the theory in question thereby achieved
is due to a new approach based on the close interaction between the strong and the vague topologies, as well as on
the theory of inner balayage, developed recently by the author.
}}
\symbolfootnote[0]{\quad 2010 Mathematics Subject Classification: Primary 31C15.}
\symbolfootnote[0]{\quad Key words: Minimum Riesz energy problems, external fields, inner capacitary measures, inner balayage.
}

\vspace{6pt}

\markboth{\emph{Natalia Zorii}} {\emph{Minimum Riesz energy problems with external fields}}

\section{Statement of the problem}\label{sec-intr}

\subsection{Basic assumptions and general facts}\label{sec-gen} Fix $n\geqslant2$ and $0<\alpha<n$. The present paper deals with minimum energy problems with respect to the $\alpha$-Riesz kernel $\kappa_\alpha(x,y):=|x-y|^{\alpha-n}$ on $\mathbb R^n$, evaluated in the presence of external fields $f:\mathbb R^n\to[-\infty,\infty]$. (Here $|x-y|$ denotes the Euclidean distance between $x,y\in\mathbb R^n$.)

Denote by $\mathfrak M=\mathfrak M(\mathbb R^n)$ the linear space of all (scalar real-valued Radon) measures $\mu$ on $\mathbb R^n$, equipped with the {\it vague} ({\it $=$\/weak\/$^*$}) topology of pointwise convergence on the class $C_0(\mathbb R^n)$ of all continuous functions $\varphi:\mathbb R^n\to\mathbb R$ of compact support, and by $\mathfrak M^+=\mathfrak M^+(\mathbb R^n)$ the cone of all positive $\mu\in\mathfrak M$, where $\mu$ is {\it positive} if and only if $\mu(\varphi)\geqslant0$ for all positive $\varphi\in C_0(\mathbb R^n)$. When speaking of a (signed) measure $\mu\in\mathfrak M$, we always understand that its $\alpha$-Riesz {\it potential}
\[U^\mu(x):=\int\kappa_\alpha(x,y)\,d\mu(y),\quad x\in\mathbb R^n,\]
is well defined and finite almost everywhere (a.e.)\ with respect to the Lebesgue measure on $\mathbb R^n$; or equivalently, that (cf.\ \cite[Section~I.3.7]{L})
\begin{equation}\label{1.3.10}
\int_{|y|>1}\,\frac{d|\mu|(y)}{|y|^{n-\alpha}}<\infty,
\end{equation}
where $|\mu|:=\mu^++\mu^-$, $\mu^+$ and $\mu^-$ being the positive and negative parts of $\mu$ in the Hahn--Jor\-dan decomposition. (This would necessarily hold if $\mu$ were assumed to be {\it bounded}, that is, with $|\mu|(\mathbb R^n)<\infty$.)\ Actually, then (and only then) $U^\mu$ is finite {\it quasi-everywhere} ({\it q.e.})\ on $\mathbb R^n$, namely, everywhere except for a set of zero outer capacity, cf.\ \cite[Section~III.1.1]{L}. (Regarding the concepts of {\it outer} and {\it inner} $\alpha$-Riesz capacities, denoted by $c^*(\cdot)$ and $c_*(\cdot)$, respectively, see \cite[Section~II.2.6]{L}.)

Yet another assumption sufficient for $U^\mu$ to be finite q.e.\ on $\mathbb R^n$ is that $\mu$ be of finite $\alpha$-Riesz {\it energy} (see \cite[Corollary to Lemma~3.2.3]{F1}), i.e.
\[I(\mu):=\int\kappa_\alpha(x,y)\,d(\mu\otimes\mu)(x,y)<\infty.\]

A basic fact to be permanently used in what follows is that the $\alpha$-Riesz kernel is {\it strictly positive definite}, which means that the energy of any (signed) $\mu\in\mathfrak M$ is ${}\geqslant0$ (whenever defined), and it is zero only for $\mu=0$, see \cite[Theorem~1.15]{L}. Then all $\mu\in\mathfrak M$ with $I(\mu)<\infty$ form a pre-Hil\-bert space $\mathcal E=\mathcal E(\mathbb R^n)$ with the inner product
\[\langle\mu,\nu\rangle:=I(\mu,\nu):=\int\kappa_\alpha(x,y)\,d(\mu\otimes\nu)(x,y)\]
and the energy norm $\|\mu\|:=\sqrt{I(\mu)}$, cf.\ \cite[Section~3.1]{F1}. The (Hausdorff) topology on $\mathcal E$ determined by means of this norm $\|\cdot\|$ is said to be {\it strong}.

Another fact crucial to our study is that the cone $\mathcal E^+=\mathcal E^+(\mathbb R^n):=\mathcal E\cap\mathfrak M^+(\mathbb R^n)$ is {\it complete} in the induced strong topology, and that the strong topology on $\mathcal E^+$ is {\it finer} than the vague topology on $\mathcal E^+$ (see \cite{Ca1,D1,De2}, cf.\ \cite[Section~I.4.13]{L}). Thus any strong Cauchy sequence (net) $(\mu_j)\subset\mathcal E^+$ converges {\it both strongly and vaguely} to the same (unique) $\mu_0\in\mathcal E^+$. The $\alpha$-Riesz kernel is, therefore, {\it perfect} \cite[Section~3.3]{F1}.\footnote{It is worth noting here that according to a well-known counterexample by H.~Cartan \cite{Ca1}, the whole pre-Hilbert space $\mathcal E$ is strongly {\it incomplete} (cf.\ \cite[Theorem~1.19]{L}).}

\subsection{Statement of the problem} For any $A\subset\mathbb R^n$, let $\mathfrak M^+(A)$ stand for the class of all $\mu\in\mathfrak M^+$ {\it concentrated on} $A$, which means that $A^c:=\mathbb R^n\setminus A$ is $\mu$-negligible, or equivalently, that $A$ is $\mu$-measurable and $\mu=\mu|_A$, $\mu|_A:=1_A\cdot\mu$ being the {\it trace} of $\mu$ to $A$, cf.\ \cite[Section~V.5.7]{B2}. (Note that for $\mu\in\mathfrak M^+(A)$, the indicator function $1_A$ of $A$ is locally $\mu$-int\-egr\-able.) We also denote
\begin{gather*}
\breve{\mathfrak M}^+(A):=\bigl\{\mu\in\mathfrak M^+(A):\ \mu(\mathbb R^n)=1\bigr\},\\
\mathcal E^+(A):=\mathcal E\cap\mathfrak M^+(A),\quad\breve{\mathcal E}^+(A):=\mathcal E\cap\breve{\mathfrak M}^+(A).
\end{gather*}

Given $A\subset\mathbb R^n$, fix a universally measurable function $f:\overline{A}\to[-\infty,\infty]$, to be treated as an external field acting on charges (measures) carried by $\overline{A}$ $\big({}:={\rm Cl}_{\mathbb R^n}A\bigr)$, and let $\mathcal E^+_f(A)$ stand for the class of all $\mu\in\mathcal E^+(A)$ such that $f$ is $\mu$-in\-te\-g\-r\-ab\-le \cite{B2} (Chapter~IV, Sections~3, 4).
Then {\it the Gauss functional} ({\it $=$\/the $f$-wei\-gh\-t\-ed energy})
\begin{equation*}\label{fen}
 I_f(\mu):=I(\mu)+2\int f\,d\mu
\end{equation*}
is finite for all $\mu\in\mathcal E^+_f(A)$, and hence one can introduce the extremal value\footnote{As usual, the infimum over the empty set is interpreted as $+\infty$. We also agree that $1/(+\infty)=0$ and $1/0 = +\infty$.}
\begin{equation*}\label{wf}
w_f(A):=\inf_{\mu\in\breve{\mathcal E}^+_f(A)}\,I_f(\mu)\in[-\infty,\infty],
\end{equation*}
where
\[\breve{\mathcal E}^+_f(A):=\mathcal E^+_f(A)\cap\breve{\mathfrak M}^+(A).\]

$\bullet$ Throughout the present paper, we always assume that
\begin{equation}\label{fin}
-\infty<w_f(A)<\infty.
\end{equation}
(See Lemma~\ref{l-aux-3} below
for necessary and sufficient conditions for (\ref{fin}) to occur.) Then the class $\breve{\mathcal E}^+_f(A)$ is nonempty, and hence the following problem makes sense.

\begin{problem}\label{pr-main} Does there exist $\lambda_{A,f}\in\breve{\mathcal E}^+_f(A)$ with $I_f(\lambda_{A,f})=w_f(A)$?\,\footnote{If the {\it unweighted} case $f=0$ takes place, then we drop the index $f$, and write $w(A)$, $\lambda_A$, etc.\ in place of $w_f(A)$ and $\lambda_{A,f}$, respectively. (Note that, actually, $c_*(A)=1/w(A)$.)}
\end{problem}

This problem, originated by C.F.~Gauss \cite{Gau} for charges on the boundary surface of a bounded domain in $\mathbb R^3$,
is sometimes referred to as {\it the inner Gauss variational problem} \cite{O}. Recent results on this topic are reviewed in the monographs \cite{BHS,ST} (see also numerous references therein); for the latest works on Problem~\ref{pr-main}, see \cite{CSW,Dr0}.

Unlike the previous studies on Problem~\ref{pr-main}, in the present paper we do not require an external field $f$ to be necessarily lower semicontinuous, and a set $A$ to be necessarily closed. The $f$-we\-i\-g\-h\-t\-ed energy $I_f(\mu)$ is, therefore, no longer vaguely lower semicontinuous, and the class $\mathfrak M^+(A)$ is no longer vaguely closed. To overcome these additional difficulties, we initiate a new approach based on the close interaction between the strong and the vague topologies on the (strongly complete) cone $\mathcal E^+(\mathbb R^n)$.

In the particular case where $0<\alpha\leqslant2$, the analysis performed below is also substantially based on the theory of {\it inner} $\alpha$-Riesz balayage, established in the author's recent papers \cite{Z-bal}--\cite{Z-arx}. (Such a theory generalizes the theory of inner Newtonian balayage, originated in the pioneering work by Cartan \cite{Ca2}. For the theory of {\it outer} $\alpha$-Riesz balayage, see the monographs \cite{BH,Doob}, the latter dealing with $\alpha=2$.)

The approach thereby developed enables us to obtain necessary and/or sufficient conditions for the solvability of Problem~\ref{pr-main} for {\it noncompact} (and {\it even nonclosed\/}) sets $A\subset\mathbb R^n$ and for quite general ({\it not necessarily lower semicontinuous}) external fields $f$ (see Theorems~\ref{th2'}, \ref{th3'}, \ref{th-nonth}, \ref{th-th} and Corollaries~\ref{qcomp}, \ref{cor2}).

Furthermore, we provide a number of alternative characterizations of the solution $\lambda_{A,f}$ to Problem~\ref{pr-main} (see Theorems~\ref{th-ch2} and \ref{th3'}), and we analyze the continuity of $\lambda_{A,f}$ as well as that of the so-cal\-led inner modified Robin constant for monotone families of sets (see Theorems~\ref{th4'}--\ref{th4'c2}, \ref{conv6}, and \ref{conv7}).

Finally, we establish a complete description of the support of $\lambda_{A,f}$ (Theorems~\ref{th-desc} and \ref{th-desc2}),
thereby giving an answer to Open question~2.1 by M.~Ohtsuka \cite[p.~284]{O}.

These improve significantly recent results on the topic in question, see Section~\ref{sec-remark} for details. New phenomena thereby discovered are illustrated by means of examples.

\subsection{General assumptions on external fields. Preliminary results on Problem~\ref{pr-main}}\label{subs-prel} For closed $F\subset\mathbb R^n$, we denote by $\Phi(F)$ the class of all lower semicontinuous (l.s.c.)\ functions $g:F\to(-\infty,\infty]$ that
are ${}\geqslant0$ unless $F$ is compact.

$\bullet$ Given {\it arbitrary} $A\subset\mathbb R^n$, in all that follows we assume $f$ to be of the form(\footnote{In the previous researches on Problem~\ref{pr-main} (see e.g.\ \cite{BHS,CSW,Dr0,ST} and references therein), an external field was always defined to be a lower semicontinuous function $\psi$, whereas the presence of an al\-t\-er\-n\-a\-t\-ive/ad\-d\-i\-t\-io\-nal source of energy, generated by {\it signed\/} charges $\vartheta\in\mathcal E(\mathbb R^n)$ and/or $\omega\in\mathfrak M(\mathbb R^n)$, cf.\ (\ref{f1}), agrees well with the original electrostatic nature of the Gauss variational problem.})(\footnote{Each of the summands in (\ref{f1}) might be $0$. (For $\omega=0$, we have $S_\omega=\varnothing$, hence $d(S_\omega,A)=+\infty$.)})
\begin{equation}\label{f1}
 f=\psi+U^\vartheta+U^\omega,
\end{equation}
where $\psi$, $\vartheta$, and $\omega$ have the following properties:
\begin{itemize}
\item[$(\mathcal P_1)$] $\psi\in\Phi(\overline{A})$.
\item[$(\mathcal P_2)$] $\vartheta\in\mathcal E(\mathbb R^n)$.
\item[$(\mathcal P_3)$] $\omega\in\mathfrak M(\mathbb R^n)$ {\it is bounded and such that
  \begin{equation*}\label{d}
    d(S_\omega,A):=\inf_{x\in S_\omega, \ y\in A}\,|x-y|>0,
  \end{equation*}
$S_\omega=S(\omega)$ being the support of\/ $\omega$.}
\end{itemize}

Such an external field $f$ is well defined as a finite number or $\pm\infty$ q.e.\ on $\overline{A}$, for so are all the summands in (\ref{f1}), see Section~\ref{sec-gen}. (Here we have used the countable subadditivity of outer capacity on arbitrary sets in $\mathbb R^n$ \cite[Section~II.2.6]{L}.)

Similarly, {\it the $f$-wei\-g\-h\-t\-ed potential}
\[U^\mu_f:=U^\mu+f,\quad\mu\in\mathfrak M,\]
is well defined as a finite number or $\pm\infty$ q.e.\ on $\overline{A}$~---  provided, of course, that the measure $\mu$ meets (\ref{1.3.10}) (thus in particular if $\mu$ is bounded or of finite energy).

\begin{lemma}\label{int}
Given an external field $f$ of form {\rm(\ref{f1})}, $\int f\,d\mu$ as well as $I_f(\mu)$ is well defined as a finite number or $+\infty$
for all bounded $\mu\in\mathcal E^+(\overline{A})$.
\end{lemma}

\begin{proof}
Observe that $f$ is $\mu$-measurable for all $\mu\in\mathfrak M^+(\overline{A})$, for so are all the summands $\psi$, $U^\vartheta$, and $U^\omega$ in (\ref{f1}). We further note that the inequality $\int\psi\,d\mu>-\infty$ is obvious whenever $\psi\geqslant0$, while the remaining case of compact $\overline{A}$ is treated by replacing $\psi$ by $\psi':=\psi+c\geqslant0$ on $\overline{A}$, where $c\in(0,\infty)$, which is always possible since a lower semicontinuous function on a compact set is lower bounded.

As $\int U^\vartheta\,d\mu$ is certainly finite for all $\mu\in\mathcal E$, it remains to verify that
\[\int U^{\omega^-}\,d\mu<\infty\text{ \ for all bounded $\mu\in\mathfrak M^+(\overline{A})$},\]
which however follows immediately from the estimates (cf.\ $(\mathcal P_3)$)
\begin{equation}\label{estme}
\int|x-y|^{\alpha-n}\,d(\omega^-\otimes\mu)(x,y)\leqslant d(S_\omega,A)^{\alpha-n}\omega^-(\mathbb R^n)\mu(\mathbb R^n)<\infty
\end{equation}
by making use of Lebesgue--Fubini's theorem \cite[Section~V.8, Theorem~1(a)]{B2}.
\end{proof}

Lemma~\ref{int} enables us to deduce from our earlier paper \cite{Z5a} the following two theorems on Problem~\ref{pr-main}. (In view of their important role in the proofs of the main results of the present study, we provide here their explicit formulations.)

\begin{theorem}[{\rm cf.\ \cite[Theorems~1, 2]{Z5a}}]\label{th-ch1}
If the solution $\lambda=\lambda_{A,f}$ to Problem~{\rm\ref{pr-main}} exists,\footnote{The solution $\lambda_{A,f}$ to Problem~\ref{pr-main} is {\it unique} (if it exists), which follows easily from the convexity of the class $\breve{\mathcal E}^+_f(A)$ by use of the parallelogram identity in the pre-Hil\-bert space $\mathcal E$ (cf.\ \cite[Lemma~6]{Z5a}). This $\lambda_{A,f}$ will also be referred to as {\it the inner $f$-weighted equilibrium measure}.} then its $f$-wei\-ghted potential $U_f^{\lambda}$ has the properties\footnote{The abbreviation {\it n.e.}\ ({\it nearly everywhere}) means, as usual, that the set of all $x\in A$ where the inequality fails is of zero {\it inner} capacity. Compare with the concept of quasi-everywhere, where an exceptional set must be of zero {\it outer} capacity (see Section~\ref{sec-gen}).}
\begin{align}
  U_f^{\lambda}&\geqslant c_{A,f}\text{ \ n.e.\ on $A$},\label{t1-pr1}\\
  U_f^{\lambda}&=c_{A,f}\text{ \ $\lambda$-a.e.\ on $A$},\label{t1-pr2}
\end{align}
where
\begin{equation}\label{cc}
c_{A,f}:=\int U_f^{\lambda}\,d\lambda=w_f(A)-\int f\,d\lambda\in(-\infty,\infty)
\end{equation}
is said to be the inner $f$-weighted equilibrium constant.\footnote{Similarly to \cite[p.~27]{ST}, $c_{A,f}$ is also said to be  {\it the inner modified Robin constant}.}
If moreover $f$ is l.s.c.\ on $\overline{A}$, then also
\begin{equation*}\label{old}U_f^{\lambda}\leqslant c_{A,f}\text{ \ on $S(\lambda)$},\end{equation*}
and hence
\[U_f^{\lambda}=c_{A,f}\text{ \ n.e.\ on $A\cap S(\lambda)$}.\]
\end{theorem}

Relations (\ref{t1-pr1}) and (\ref{cc}), resp.\ (\ref{t1-pr2}) and (\ref{cc}), characterize the minimizer $\lambda_{A,f}$ uniquely within $\breve{\mathcal E}^+_f(A)$. In more detail, the following theorem holds true.

\begin{theorem}[{\rm cf.\ \cite[Proposition~1]{Z5a}}]\label{th-ch2}For $\mu\in\breve{\mathcal E}^+_f(A)$ to be the {\rm(}unique{\rm)} solution $\lambda_{A,f}$ to Problem~{\rm\ref{pr-main}}, it is necessary and sufficient that either of the following two characteristic inequalities be fulfilled:
\begin{gather}U_f^\mu\geqslant\int U_f^\mu\,d\mu=:c_1\text{ n.e.\ on $A$},\label{1}\\
U_f^\mu\leqslant w_f(A)-\int f\,d\mu=:c_2\text{ \ $\mu$-a.e.\ on $A$.}\label{2}
\end{gather}
If either of {\rm(\ref{1})} or {\rm(\ref{2})} holds, then equality actually prevails in {\rm(\ref{2})}, and moreover
\[c_1=c_2=c_{A,f},\] $c_{A,f}$ being the inner $f$-weighted equilibrium constant.
\end{theorem}

\subsection{When does (\ref{fin}) hold?} For (\ref{fin}) to be satisfied, it is necessary that $c_*(A)>0$, since otherwise the class $\mathcal E^+(A)$ would reduce to $\{0\}$ (see e.g.\ \cite[Lemma~2.3.1]{F1}). Actually, the following stronger assertion holds true.

\begin{lemma}\label{l-aux-3} For {\rm(\ref{fin})} to hold, it is necessary and sufficient that
\begin{equation}\label{iv;}
c_*\bigl(\{x\in A:\ \psi(x)<\infty\}\bigr)>0,
\end{equation}
$\psi$ being the first summand in representation {\rm(\ref{f1})}.
\end{lemma}

\begin{proof}
According to \cite[Lemma~5]{Z5a}, $w_f(A)<\infty$ is fulfilled if and only if
\[c_*\bigl(\{x\in A:\ |f|(x)<\infty\}\bigr)>0,\]
which however is equivalent to (\ref{iv;}) since the second and the third summands in (\ref{f1}) take finite values q.e.\ (hence n.e.) on $\mathbb R^n$ (see Section~\ref{sec-gen}).\footnote{Here we have used the following strengthened version of countable subadditivity for inner capacity: {\it For arbitrary $A\subset\mathbb R^n$ and universally measurable $U_j\subset\mathbb R^n$, $j\in\mathbb N$,
\[c_*\Bigl(\bigcup_{j\in\mathbb N}\,A\cap U_j\Bigr)\leqslant\sum_{j\in\mathbb N}\,c_*(A\cap U_j).\]} (See \cite[pp.~157--159]{F1} or \cite[p.~253]{Ca2}, compare with \cite[p.~144]{L}.)\label{foot-sub}}

It thus remains to show that $w_f(A)>-\infty$. For any given $\vartheta\in\mathcal E$,
\begin{equation*}\label{Isigma}
I_{U^\vartheta}(\mu)=\|\mu\|^2+2\int U^\vartheta\,d\mu=\|\mu+\vartheta\|^2-\|\vartheta\|^2\text{ \ for all $\mu\in\mathcal E$};
\end{equation*}
hence, by the strict positive definiteness of the $\alpha$-Riesz kernel,
\[w_{U^\vartheta}(A)\geqslant-\|\vartheta\|^2>-\infty.\]
On account of (\ref{estme}), the proof will be completed by verifying the inequality
\[\inf_{\mu\in\breve{\mathcal E}^+(A)}\,\int\psi\,d\mu>-\infty.\]
This however is obvious if $\psi\geqslant0$, while the remaining case of compact $\overline{A}$ is treated as usual, namely, by replacing $\psi$ by $\psi':=\psi+c\geqslant0$ on $\overline{A}$, where $c\in(0,\infty)$.
\end{proof}

\section{Main results}\label{sec-intr2}

$\bullet$ Recall that throughout the whole paper,
an external field $f$ is of form (\ref{f1}), where $\psi$, $\vartheta$, and $\omega$ satisfy $(\mathcal P_1)$--$(\mathcal P_3)$ as well as (\ref{iv;}). (These general conventions will not be repeated hereafter.)

$\bullet$ Along with these general assumptions,  throughout Sections~\ref{subs-main2}--\ref{sec-descr} as well as in Theorem~\ref{th2'}, we additionally require a set $A\subset\mathbb R^n$ to have the following property:
\begin{itemize}
  \item[$(\mathcal P_4)$] {\it The class $\mathcal E^+(A)$ is closed in the strong topology on $\mathcal E^+$.} (This in particular occurs if $A$ is {\it quasiclosed} ({\it quasicompact\/}), that is, if $A$ can be approximated in outer capacity by closed (compact) sets \cite{F71}; see Theorem~\ref{l-quasi} below.)
\end{itemize}

\subsection{The case of $\alpha\in(0,n)$}\label{subs-main} Theorems~\ref{th2'}, \ref{th4'}--\ref{th4'c2} and Corollary~\ref{qcomp} do hold for the $\alpha$-Riesz kernels of {\it arbitrary} order $\alpha\in(0,n)$.

\begin{theorem}\label{th2'} Assume, in addition, that $(\mathcal P_4)$ is fulfilled. If moreover
\begin{equation}\label{cafi'}
c_*(A)<\infty,
\end{equation}
then the {\rm(}unique{\rm)} solution $\lambda_{A,f}$ to Problem~{\rm\ref{pr-main}} does exist.\end{theorem}

\begin{corollary}\label{qcomp}
Problem~{\rm\ref{pr-main}} is solvable for any quasicompact $A\subset\mathbb R^n$.
\end{corollary}

\begin{proof}
For quasicompact $A$, (\ref{cafi'}) necessarily holds, the $\alpha$-Riesz capacity of a compact set being obviously finite. Since
$(\mathcal P_4)$ is fulfilled as well (Theorem~\ref{l-quasi}), the corollary follows directly from Theorem~\ref{th2'}.
\end{proof}

Given $A\subset\mathbb R^n$, denote by $\mathfrak C_A$ the upward directed set of all compact subsets $K$ of $A$, where $K_1\leqslant K_2$ if and only if $K_1\subset K_2$. If a net $(x_K)_{K\in\mathfrak C_A}\subset Y$ converges to $x_0\in Y$, $Y$ being a topological space, then we shall indicate this fact by writing
\begin{equation*}x_K\to x_0\text{ \ in $Y$ as $K\uparrow A$}.\end{equation*}

\begin{theorem}\label{th4'} Assume that the solution $\lambda_{A,f}$ to Problem~{\rm\ref{pr-main}} exists.\footnote{See Theorems~\ref{th2'}, \ref{th3'}, \ref{th-nonth} and Corollaries~\ref{qcomp}, \ref{cor2}(b) for sufficient conditions for this to hold.} Then
\begin{equation}\label{conv2}\lambda_{K,f}\to\lambda_{A,f}\text{ \ strongly and vaguely in $\mathcal E^+$ as
$K\uparrow A$},\end{equation}
$\lambda_{K,f}$ being the solution to Problem~{\rm\ref{pr-main}} for $K\in\mathfrak C_A$ large enough.
The inner $f$-wei\-g\-h\-t\-ed equilibrium constant $c_{K,f}$ also varies continuously when $K\uparrow A$, that is,
\begin{equation}\label{conv3}\lim_{K\uparrow A}\,c_{K,f}=c_{A,f}.\end{equation}
\end{theorem}

\begin{theorem}\label{th4'c}Let $(A_j)\subset\mathbb R^n$ be an increasing sequence of universally measurable sets with the union $A$. Then
\begin{equation}\label{cont-count1}
w_f(A_j)\downarrow w_f(A)\text{ \ as $j\to\infty$}.
\end{equation}
If moreover there exist the minimizers $\lambda_{A,f}$ and $\lambda_{A_j,f}$, then also\footnote{This holds, for instance, if $A_j$, $j\in\mathbb N$, are closed, and $A$, their union, is of finite capacity. See also Theorem~\ref{conv6} below.\label{th4'cf}}
\begin{gather}
\label{cont-count12}\lambda_{A_j,f}\to\lambda_{A,f}\text{ \ strongly and vaguely in $\mathcal E^+$ as $j\to\infty$},\\
\label{cont-count13}\lim_{j\to\infty}\,c_{A_j,f}=c_{A,f}.
\end{gather}
\end{theorem}

\begin{theorem}\label{th4'c2}Let $(A_j)\subset\mathbb R^n$ be a decreasing sequence of quasiclosed sets with the intersection $A$, and let $c_*(A_{j_0})<\infty$ for some $j_0\in\mathbb N$. Then
\begin{equation}\label{cont-count1d}
w_f(A_j)\uparrow w_f(A)\text{ \ as $j\to\infty$},
\end{equation}
and moreover {\rm(\ref{cont-count12})}--{\rm(\ref{cont-count13})} hold true {\rm(}the minimizers $\lambda_{A_j,f}$, $j\geqslant j_0$, and $\lambda_{A,f}$ do exist{\rm)}.\end{theorem}

\begin{remark}
Theorem~\ref{th4'c2} as well as Theorem~\ref{conv7} (see below) remains valid for {\it nets} in place of sequences, provided that the sets in question are {\it closed}.
\end{remark}

\subsection{The case of $\alpha\in(0,2]$}\label{subs-main2}Throughout Sections~\ref{subs-main2}--\ref{sec-descr}, we always assume that:
\begin{itemize}
\item $\alpha\in(0,2]$.
\item $(\mathcal P_4)$ is fulfilled.
\item $\psi$, $\vartheta$, and $\omega$ in (\ref{f1}) satisfy the hypotheses
$\psi=0$, $\vartheta^+=\omega^+=0$.\footnote{Due to $\psi=0$, (\ref{iv;}) is equivalent to the assumption $c_*(A)>0$.}
\end{itemize}

To simplify notations, write $\tau:=\vartheta^-$ and $\sigma:=\omega^-$; then $f$ takes the form
\begin{equation}\label{f3'}
f=-U^\tau-U^\sigma=-U^\delta,
\end{equation}
where $\delta:=\tau+\sigma\in\mathfrak M^+(\mathbb R^n)$, while $\tau$ and $\sigma$ have the following properties:
\begin{itemize}
\item[$(\mathcal P_2')$] $\tau\in\mathcal E^+(\mathbb R^n)$.
\item[$(\mathcal P_3')$] $\sigma\in\mathfrak M^+(\mathbb R^n)$ {\it is bounded and such that}
\begin{equation}\label{f3''}d(S_\sigma,A)>0.\end{equation}
\end{itemize}

Denote by $\mu^Q$ {\it the inner $\alpha$-Riesz balayage} of $\mu$ to $Q$, where $\mu\in\mathfrak M^+$ and $Q\subset\mathbb R^n$ are {\it arbitrary} (see \cite{Z-bal,Z-bal2} for some basic results on this concept, and \cite{Z-arx1}--\cite{Z-arx} for their further development; cf.\ also Section~\ref{sec-lemmas2} below).

\begin{theorem}\label{th3'}Under the above assumptions, if moreover
\begin{equation}\label{leq1}
 \delta(\mathbb R^n)\leqslant1,
\end{equation}
then Problem~{\rm\ref{pr-main}} is solvable if and only if
\begin{equation}\label{bal1}
\text{either \ $c_*(A)<\infty$, \ or \ $\delta^A(\mathbb R^n)=1$.}
\end{equation}
In the affirmative case, the {\rm(}unique{\rm)} solution $\lambda_{A,f}$ to Problem~{\rm\ref{pr-main}} has the form
\begin{equation}\label{RR}\lambda_{A,f}=\left\{
\begin{array}{cl}\delta^A+\eta_{A,f}\gamma_A&\text{if \ $c_*(A)<\infty$},\\
\delta^A&\text{otherwise},\\ \end{array} \right.
\end{equation}
where $\gamma_A$ is the inner capacitary measure on $A$, normalized by $\gamma_A(\mathbb R^n)=c_*(A)$, while
\begin{equation}\label{eqalt}
\eta_{A,f}:=\frac{1-\delta^A(\mathbb R^n)}{c_*(A)}\in[0,\infty).
\end{equation}

Furthermore, this $\lambda_{A,f}$ can alternatively be characterized by means of any one of the following three assertions:
\begin{itemize}
\item[{\rm(i)}] $\lambda_{A,f}$ is the unique measure of minimum energy in the class
\begin{equation}\label{gamma}
\Lambda_{A,f}:=\bigl\{\mu\in\mathfrak M^+: \ U^\mu_f\geqslant\eta_{A,f}\text{ \ n.e.\ on\ $A$}\bigr\},
  \end{equation}
$\eta_{A,f}$ being introduced by formula {\rm(\ref{eqalt})}. That is, $\lambda_{A,f}\in\Lambda_{A,f}$ and
\begin{equation}\label{minen}
I(\lambda_{A,f})=\min_{\mu\in\Lambda_{A,f}}\,I(\mu).
\end{equation}
\item[{\rm(ii)}] $\lambda_{A,f}$ is the unique measure of minimum potential in the class $\Lambda_{A,f}$, introduced by means of {\rm(\ref{gamma})}. That is, $\lambda_{A,f}\in\Lambda_{A,f}$ and\footnote{This implies immediately that
$\lambda_{A,f}$ can also be characterized as the unique measure of minimum $f$-we\-i\-g\-h\-t\-ed potential in the class $\Lambda_{A,f}$ --- now, however, nearly everywhere on $\mathbb R^n$: \begin{equation*} U^{\lambda_{A,f}}_f=\min_{\mu\in\Lambda_{A,f}}\,U^\mu_f\text{ \ n.e.\ on\ $\mathbb R^n$}.\end{equation*}
This relation as well as (\ref{minen}) and (\ref{minpot}) would remain valid if $\mu\in\Lambda_{A,f}$ were of finite energy.}
\begin{equation}\label{minpot}
U^{\lambda_{A,f}}=\min_{\mu\in\Lambda_{A,f}}\,U^\mu\text{ \ on\ $\mathbb R^n$}.
\end{equation}
\item[{\rm(iii)}] $\lambda_{A,f}$ is the only measure in $\mathcal E^+(A)$ with the property
$U^{\lambda_{A,f}}_f=\eta_{A,f}$ n.e.\ on $A$.
\end{itemize}

In addition, the inner $f$-weighted equilibrium constant $c_{A,f}$, introduced by {\rm(\ref{cc})}, admits an alternative representation
\begin{equation}\label{const-alt}c_{A,f}=\eta_{A,f},\end{equation}
and hence {\rm(\ref{conv3})} can be specified as follows:
\begin{equation}\label{conv3'}c_{K,f}\downarrow c_{A,f}\text{ \ in $\mathbb R$ as $K\uparrow A$}.\end{equation}
\end{theorem}

\begin{remark}
By \cite[Corollary~5.3]{Z-bal2}, the latter relation in (\ref{bal1}) is valid e.g.\ if
\[\delta(\mathbb R^n)=1\text{\ while $A$ is {\it not} inner $\alpha$-thin at infinity.}\]
Recall that according to  \cite{KM,Z-bal2}, $Q\subset\mathbb R^n$ is said to be {\it inner $\alpha$-thin at infinity} if\footnote{For a concept of {\it outer} thinness of $Q$ at infinity when $\alpha=2$, see M.~Brelot \cite[p.~313]{Brelot} as well as J.L.~Doob \cite[pp.~175--176]{Doob}. As shown in \cite[Remark~2.3]{Z-bal2}, these two concepts are, actually, different, the latter being less restrictive. But if $Q$ is Borel, then for $\alpha=2$, the concept of inner thinness, given by (\ref{sum}), coincides, in fact, with that of outer thinness introduced by Doob. Therefore, when speaking of inner $\alpha$-thinness for Borel sets, the term "inner" might be omitted.}
 \begin{equation}\label{sum}
 \sum_{k\in\mathbb N}\,\frac{c_*(Q_k)}{q^{k(n-\alpha)}}<\infty,
 \end{equation}
where $q\in(1,\infty)$ and $Q_k:=Q\cap\{x\in\mathbb R^n:\ q^k\leqslant|x|<q^{k+1}\}$; or equivalently, if either $Q$ is bounded, or  $x=0$ is an inner $\alpha$-irregular boundary point for the inverse of $Q$ with respect to $|x|=1$. (For the concept of inner $\alpha$-irregular boundary points for arbitrary $Q\subset\mathbb R^n$, see the author's recent paper \cite[Section~6]{Z-bal}; compare with N.S.~Landkof's book \cite[Section~V.1]{L}, where $Q$ was required to be Borel.)
\end{remark}

\begin{corollary}\label{cortm}
Under the assumptions of Theorem~{\rm\ref{th3'}}, if moreover {\rm(\ref{bal1})} is fulfilled,
 then $\lambda_{A,f}$ is a measure of minimum total mass in the class $\Lambda_{A,f}$, i.e.
\begin{equation}\label{eq-t-m}\lambda_{A,f}(\mathbb R^n)=\min_{\mu\in\Lambda_{A,f}}\,\mu(\mathbb R^n)\quad\bigl({}=1\bigr).\end{equation}
\end{corollary}

\begin{remark}\label{t-m-nonun} The extremal property (\ref{eq-t-m}) cannot, however, serve as an alternative characterization of the minimizer $\lambda_{A,f}$, for it does not determine $\lambda_{A,f}$ uniquely within $\Lambda_{A,f}$. Indeed, take, for instance, $A:=\{|x|\geqslant 1\}$, and let $f$ be given by (\ref{f3'}) with $\delta\in\breve{\mathfrak M}^+(A^c)$. Since $A$ is {\it not} $\alpha$-thin at infinity, applying \cite[Corollary~5.3]{Z-bal2} gives
\begin{equation}\label{eq-t-m1}\delta^A(\mathbb R^n)=\delta(\mathbb R^n)=1.\end{equation}
Hence $\eta_{A,f}=0$, which in view of the equality $U^{\delta^A}=U^\delta$ n.e.\ on $A$  \cite[Eq.~(1.11)]{Z-bal2} yields $\delta,\delta^A\in\Lambda_{A,f}$. Noting that $\delta^A\ne\delta$, and taking (\ref{eq-t-m1}) into account, we therefore conclude that there are actually infinitely many measures of minimum total mass in $\Lambda_{A,f}$, for so are all the measures of the form $a\delta+b\delta^A$, where $a,b\in[0,1]$ and $a+b=1$.\end{remark}

\begin{corollary}\label{cor2}The following two assertions on the existence of $\lambda_{A,f}$ hold true.
\begin{itemize}
\item[{\rm(a)}] If $c_*(A)=\infty$ and $\delta(\mathbb R^n)<1$, then $\lambda_{A,f}$ fails to exist.
\item[{\rm(b)}] If $\delta\in\breve{\mathcal E}^+(A)$, then $\lambda_{A,f}$ does exist. Moreover, then
   $\lambda_{A,f}=\delta^A=\delta$, and hence
  Theorem~{\rm\ref{th3'}} is fully applicable to both $\lambda_{A,f}$ and $c_{A,f}$.
  \end{itemize}
\end{corollary}

\begin{remark}\label{f-stre}Corollary~\ref{cor2}(a) improves \cite[Corollary~2.6(ii)]{Dr0} by showing that the latter remains valid if the closed set $\Sigma$ involved in it, is just of infinite capacity. (Observe that the set $\Sigma$ in \cite{Dr0} was required throughout not to be $\alpha$-thin at infinity. Regarding the existence of a set with infinite capacity which is, nonetheless, $\alpha$-thin at infinity, see \cite[Remark~2.2]{Z-bal2}, cf.\ also Example~\ref{ex} below.)\end{remark}

\begin{theorem}\label{th-nonth}
 Assume that $A$ is not inner $\alpha$-thin at infinity. Then we have:
 \begin{itemize}
\item[{\rm(a$_1$)}] If $\delta(\mathbb R^n)=1$, $\lambda_{A,f}$ does exist. Moreover, since then $\delta^A(\mathbb R^n)=1$,  Theorem~{\rm\ref{th3'}} is fully applicable to both $\lambda_{A,f}$ and $c_{A,f}$. In particular,
    \begin{equation}\label{iff''}\lambda_{A,f}=\delta^A\text{ \ and \ }c_{A,f}=0.\end{equation}
\item[{\rm(b$_1$)}] Assume moreover that $f$ is l.s.c.\ on $\overline{A}$.\footnote{This occurs, for instance, if $S(\tau)\cap\overline{A}=\varnothing$, cf.\ (\ref{f3'}) and $(\mathcal P_3')$.} Then
 \begin{equation}\label{iff}
 \text{$\lambda_{A,f}$ exists}\iff\delta(\mathbb R^n)\geqslant1.
 \end{equation}
 In addition,\footnote{Compare with (\ref{iff''}).}
 \begin{equation}\label{F1}
   c_{A,f}<0\text{ \ if \  $\delta(\mathbb R^n)>1$}.
 \end{equation}
 \end{itemize}
\end{theorem}

\begin{remark}The requirement on $A$ of not being inner $\alpha$-thin at infinity is important for the validity of both (a$_1$) and (b$_1$) in Theorem~\ref{th-nonth} (see Theorem~\ref{th-th}).\end{remark}

Let $F\subset\mathbb R^n$ be closed and $\alpha$-thin at infinity. If moreover $\alpha=2$, then there exists the {\it unique} connected component $\Delta_F$ of the (open) set $F^c$ such that $(\Delta_F)^c$ $({}\supset F)$ still remains (closed and) $2$-thin at infinity.\footnote{Indeed, for $F$ compact, $\Delta_F$ is, in fact, the (unique) unbounded connected component of $F^c$. For $F$ noncompact, the origin $x=0$ is $2$-irregular for $F^*$, the inverse of $F\cup\{\infty_{\mathbb R^n}\}$ with respect to $|x|=1$. (Here $\infty_{\mathbb R^n}$ denotes the Alexandroff point of $\mathbb R^n$.) By \cite[Section~VIII.6, Remark~3]{Brelo1}, there exists, therefore, a unique connected component $G$ of the (open) set $(F^*)^c$ such that $x=0$ is $2$-ir\-reg\-ular for $G^c$, and the inverse $\Delta_F$ of this $G$ with respect to $|x|=1$ is as claimed.} For the given $F$, denote
\begin{equation}\label{Om}
\Omega_F=\left\{
\begin{array}{cll} F^c&\text{if}&\alpha<2,\\
\Delta_F&\text{if}&\alpha=2.\\ \end{array} \right.
\end{equation}

\begin{theorem}\label{th-th} Under the hypotheses listed at the beginning of this subsection, assume moreover that $c_*(A)=\infty$, and that $\overline{A}$ is $\alpha$-thin at infinity. Then Problem~{\rm\ref{pr-main}} is unsolvable whenever the following two assumptions are fulfilled:
\begin{equation*}\label{nons}
 \delta(\mathbb R^n)\leqslant1,\quad\delta(\Omega_{\overline{A}})>0,
\end{equation*}
$\Omega_{\overline{A}}$ being introduced by {\rm(\ref{Om})} with $F:=\overline{A}$.
\end{theorem}

\begin{remark}
As seen from either of Theorems~\ref{th2'} and \ref{th3'}, the assumption $c_*(A)=\infty$ is necessary for the validity of both Corollary~\ref{cor2}(a) and Theorem~\ref{th-th}.
\end{remark}

\subsection{Convergence results} Referring to Theorems~\ref{th4'}--\ref{th4'c2} for convergence results for $\alpha\in(0,n)$, assume now that $\alpha\leqslant2$, and that $f$ is of form (\ref{f3'}).

\begin{theorem}\label{conv6}
Let $(A_j)$ be an increasing sequence of closed set $A_j\subset\mathbb R^n$ with the union $A$ which is not $\alpha$-thin at infinity, and let each $A_j$ be either of finite capacity, or not $\alpha$-thin at infinity. Assume that either $\delta(\mathbb R^n)=1$, or $U^\delta$ is continuous on $A$ while $\delta(\mathbb R^n)>1$. Then {\rm(\ref{cont-count1})}--{\rm(\ref{cont-count13})} do hold {\rm(}the minimizers $\lambda_{A_j,f}$ and $\lambda_{A,f}$ do exist{\rm)}.\footnote{In the case where $c(A)<\infty$, limit relations (\ref{cont-count1})--(\ref{cont-count13}) hold true under much more general assumptions~--- namely, for any $\alpha\in(0,n)$ and any external field $f$ of form (\ref{f1}) (see footnote~\ref{th4'cf}). Here $c(A):=c_*(A)=c^*(A)$, Borel sets being {\it capacitable} \cite[Theorem~2.8]{L}.}
\end{theorem}

\begin{theorem}\label{conv7}
Consider a decreasing sequence $(A_j)$ of quasiclosed $A_j$ with the intersection $A$ which is not inner $\alpha$-thin at infinity, and let $f$ be of form {\rm(\ref{f3'})} with $d(S_\sigma,A_1)>0$ and $\delta(\mathbb R^n)=1$. Then {\rm(\ref{cont-count12})}--{\rm(\ref{cont-count1d})} hold true. Furthermore,
\begin{equation}\label{cont-count1du}
 U^{\lambda_{A_j,f}}\downarrow U^{\lambda_{A,f}}\text{ \ pointwise on $\mathbb R^n$ as $j\to\infty$}.
\end{equation}
\end{theorem}

\subsection{A description of $S(\lambda_{A,f})$}\label{sec-descr} The {\it reduced kernel} $\check{A}$ of $A$ is the set of all $x\in A$ with $c_*(B(x,r)\cap A)>0$ for any $r>0$, where $B(x,r):=\{y\in\mathbb R^n:\ |y-x|<r\}$, see \cite[p.~164]{L}. Under the assumptions listed at the beginning of Section~\ref{subs-main2}, a description of the support of the minimizer $\lambda_{A,f}$
is given in Theorems~\ref{th-desc} and \ref{th-desc2}.

\begin{theorem}\label{th-desc} Let $A$ be closed, $A=\check{A}$,
$\delta|_{A^c}\ne0$, and let {\rm(a$_2$)} or {\rm(b$_2$)} occur:
\begin{itemize}
  \item [{\rm(a$_2$)}] $c(A)<\infty$ and $\delta(\mathbb R^n)\leqslant1$. If $\alpha=2$, assume moreover that $A^c$ is connected.
  \item [{\rm(b$_2$)}] $A$ is not $\alpha$-thin at infinity, and $\delta(\mathbb R^n)=1$.
\end{itemize}
If\/ $D$ denotes the union of all connected components\/ $D_i$ of\/ $A^c$ with\/ $\delta|_{D_i}>0$, then
\begin{equation}\label{det5}
 S(\lambda_{A,f})=\left\{
\begin{array}{cll} A&\text{if}&\alpha<2,\\
S(\delta|_A)\cup\partial_{\mathbb R^n}D&\text{if}&\alpha=2.\\ \end{array} \right.
\end{equation}
\end{theorem}

\begin{remark}\label{O}Formula (\ref{det5}) gives an answer to \cite[p.~284, Open question~2.1]{O} (formulated for compact $A=K$) about conditions ensuring the identity $S(\lambda_{A,f})=A$.\end{remark}

\begin{remark}\label{E?}
Let $n\geqslant3$, $A:=\overline{\mathbb R^{n-1}}$ $({}:={\rm Cl}_{\mathbb R^n}\,\mathbb R^{n-1})$, $\mathbb R^{n-1}$ being immersed in $\mathbb R^n$, and let $f:=-U^{\varepsilon_{x_0}}$, where $\varepsilon_{x_0}$ is the unit Dirac measure at $x_0\in\mathbb R^n\setminus\overline{\mathbb R^{n-1}}$. As stated in \cite[Section~4.2, case~(i)]{Dr0}, $S(\lambda_{A,f})=\overline{\mathbb R^{n-1}}$. However, for $\alpha=2$, the quoted assertion from \cite{Dr0} is false. (Note that, by (\ref{det5}), $S(\lambda_{A,f})=\partial\mathbb R^{n-1}:=\partial_{\mathbb R^n}\mathbb R^{n-1}$.)

To substantiate this, we note that $\lambda_{A,f}$ is the Newtonian balayage $\varepsilon_{x_0}'$ of $\varepsilon_{x_0}$ onto $\overline{\mathbb R^{n-1}}$, see the latter formula in representation (\ref{RR}). Denoting by $K_0$ the inverse of the one-point compactification of $\overline{\mathbb R^{n-1}}$ with respect to $|x-x_0|=r$, $r>0$ being small enough, we further note that $\varepsilon_{x_0}'$ is the Kelvin transform of the Newtonian capacitary measure $\gamma_{K_0}$ on $K_0$ (see \cite[Section~IV.6.24]{L}). Since $K_0$ is compact and coincides with its reduced kernel, while its complement is connected, we have $S(\gamma_{K_0})=\partial K_0$ \cite[Section~II.3.13]{L}, which implies that, indeed, $S(\lambda_{A,f})=\partial\mathbb R^{n-1}$ $({}\ne\overline{\mathbb R^{n-1}})$.
\end{remark}

\begin{theorem}\label{th-desc2}  Assume that $A$ is not inner $\alpha$-thin at infinity, $\delta(\mathbb R^n)>1$, $U^\delta$ is continuous on $\overline{A}$, and there is the limit
\begin{equation}\label{limm}
\lim_{x\to\infty_{\mathbb R^n}, \ x\in A}\,U^\delta(x).
\end{equation}
Then $\lambda_{A,f}$ is of compact support.
\end{theorem}

\begin{remark}
Theorem~\ref{th-desc2} is {\it sharp} in the sense that $\delta(\mathbb R^n)>1$ cannot in general be replaced by $\delta(\mathbb R^n)=1$. Indeed, if $A$ is closed and not $\alpha$-thin at infinity, $\delta(\mathbb R^n)=1$, and $\delta(A^c)>0$, then for $\alpha<2$, $S(\lambda_{A,f})$ is {\it always noncompact}, cf.\ (\ref{det5}).
\end{remark}

\begin{figure}[htbp]
\begin{center}
\vspace{-.2in}
\hspace{-.1in}\includegraphics[width=4.6in]{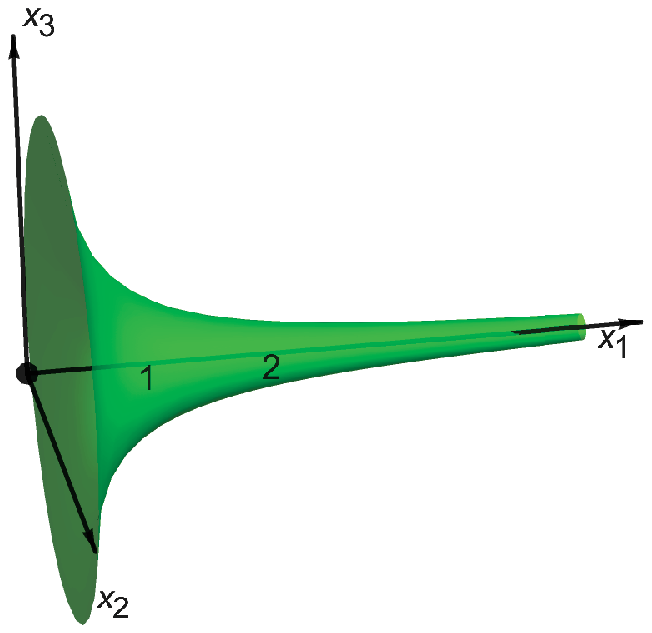}
\vspace{-1.8in}
\caption{The set $F_1$ in Example~\ref{ex} with $\varrho_1(x_1)=1/x_1$.\vspace{-.0in}}
\label{Fig2}
\end{center}
\end{figure}

\begin{example}\label{ex} On $\mathbb R^3$, consider the kernel $1/|x-y|$ and the rotation bodies
\begin{equation*}F_i:=\bigl\{x\in\mathbb R^3: \ 0\leqslant x_1<\infty, \
x_2^2+x_3^2\leqslant\varrho_i^2(x_1)\bigr\}, \ i=1,2,3,\end{equation*}
where
\begin{align*}
\varrho_1(x_1)&:=x_1^{-s}\text{ \ with\ }s\in[0,\infty),\\
\varrho_2(x_1)&:=\exp(-x_1^s)\text{ \ with\ }s\in(0,1],\\
\varrho_3(x_1)&:=\exp(-x_1^s)\text{ \ with\ }s\in(1,\infty).
\end{align*}
As seen from estimates in \cite[Section~V.1, Example]{L}, $F_1$ is not $2$-thin at infinity, $F_2$ is $2$-thin at infinity, though has infinite Newtonian capacity, whereas $F_3$ is of finite Newtonian capacity.
Therefore, $\lambda_{F_3,f}$ exists for any $f$ of form (\ref{f1}) (Theorem~\ref{th2'}).

Let now $f$ be of form (\ref{f3'}), where $S(\delta)\subset F_i^c$ is compact. Then $\lambda_{F_1,f}$ exists if and only if $\delta(\mathbb R^n)\geqslant1$ (Theorem~\ref{th-nonth}(b$_1$)). If moreover $\delta(\mathbb R^n)=1$, then $S(\lambda_{F_1,f})=\partial F_1$ (Theorem~\ref{th-desc}(b$_2$)) while $c_{F_1,f}=0$ (Theorem~\ref{th-nonth}(a$_1$)); whereas for $\delta(\mathbb R^n)>1$, $S(\lambda_{F_1,f})$ is compact (Theorem~\ref{th-desc2}) while $c_{F_1,f}<0$ (Theorem~\ref{th-nonth}(b$_1$)). If now $\delta(\mathbb R^n)\leqslant1$, then $S(\lambda_{F_3,f})=\partial F_3$  (Theorem~\ref{th-desc}(a$_2$)), whereas $\lambda_{F_2,f}$ fails to exist (Theorem~\ref{th-th}).
\end{example}

\begin{figure}[htbp]
\begin{center}
\vspace{-.3in}
\hspace{-1.1in}\includegraphics[width=4.0in]{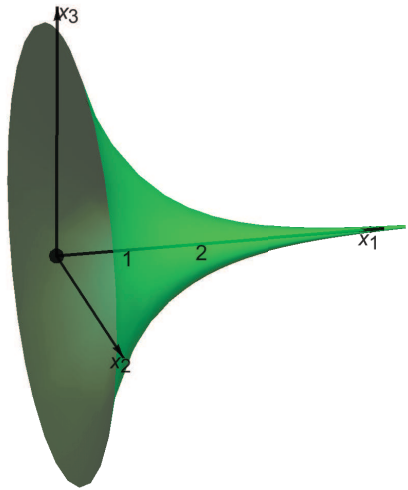}
\vspace{-1.0in}
\caption{The set $F_2$ in Example~\ref{ex} with $\varrho_2(x_1)=\exp(-x_1)$.\vspace{-.1in}}
\label{Fig1}
\end{center}
\end{figure}

\subsection{Remark}\label{sec-remark}
  In the stated generality,
  all the results thus obtained seem to be new; however, a few of them were known before for some very particular $f$ and/or $A$. For instance, if $\alpha\in(0,2]$, $A$ is closed while $\delta$, the measure in representation (\ref{f3'}), is just $q\varepsilon_{x_0}$, where $q\in(0,\infty)$ and $x_0\not\in A$, then (\ref{iff}) and Theorem~\ref{th-desc2}
  were given in \cite[Corollary~2.6]{Dr0} (cf.\ also \cite[Theorem~2.1]{BDO}). See also Remarks~\ref{f-stre} and \ref{E?} above. Such a significant improvement of the theory in question is due to a new approach, based on the close interaction between the strong and the vague topologies on the strongly complete cone $\mathcal E^+$, as well as on the author's recent theory of inner balayage.

\section{Preliminary results}\label{sec-aux}

\subsection{On convergent nets} We shall first establish several auxiliary results related to convergent nets $(\mu_s)_{s\in S}$ of positive measures.\footnote{Since $\mathfrak M(\mathbb R^n)$ equipped with the vague topology satisfies the first axiom of countability (see e.g.\ \cite[Lemma~4.4]{Z-arx}), it is usually enough to consider convergent {\it sequences} of measures. Nonetheless, in the study of inner potential theoretical concepts, it is often convenient to operate with {\it nets}, in particular when dealing with the upward directed set $\mathfrak C_A$ of all compact subsets $K$ of $A$.}

\begin{lemma}\label{lemma-semi}
Assume that a set $F\subset\mathbb R^n$ is closed, a function $g:F\to(-\infty,\infty]$ is l.s.c.\ and lower bounded, and a net
$(\mu_s)_{s\in S}\subset\mathfrak M^+(F)$ converges vaguely to $\mu_0$.\footnote{Then necessarily $\mu_0\in\mathfrak M^+(F)$, $\mathfrak M^+(F)$ being vaguely closed \cite[Section~III.2, Proposition~6]{B2}.} If moreover either $g\geqslant0$ on $F$, or
  \begin{equation}\label{mu}
    \lim_{s\in S}\,\mu_s(\mathbb R^n)=\mu_0(\mathbb R^n)\in(0,\infty),
  \end{equation}
then
\begin{equation}\label{eq-lower}
 \int g\,d\mu_0\leqslant\liminf_{s\in S}\,\int g\,d\mu_s.
\end{equation}
\end{lemma}

\begin{proof}
For $g\geqslant0$, this is well known (see \cite[Section~IV.1, Proposition~4]{B2}, applied to $F$ treated as a locally compact subspace of $\mathbb R^n$). Otherwise, there exists $c\in(0,\infty)$ such that $g':=g+c$ is l.s.c.\ and positive on $F$. Applying (\ref{eq-lower}) to $g'$, and then subtracting (\ref{mu}) multiplied by $c$
from the inequality thereby obtained, we get the lemma.\end{proof}

\begin{remark}
 Assumption (\ref{mu}) is certainly fulfilled if $F=K$ is compact, for the mapping $\mu\mapsto\mu(K)$ is vaguely continuous on $\mathfrak M^+(K)$. Yet another possibility for (\ref{mu}) to hold is described in Lemma~\ref{str-cl-1} below.
\end{remark}

\begin{lemma}\label{str-cl-1}
Given an arbitrary set $A\subset\mathbb R^n$ with $c_*(A)<\infty$, consider a net $(\mu_s)_{s\in S}\subset\breve{\mathcal E}^+(A)$ converging strongly {\rm(}hence vaguely{\rm)} to $\mu_0\in\mathcal E^+$. Then
\begin{equation}\label{exm}\mu_0(\mathbb R^n)=1.\end{equation}
\end{lemma}

\begin{proof} Taking a subnet if necessary and changing notation, we can certainly assume $(\mu_s)_{s\in S}$ to be strongly bounded:
\begin{equation}\label{str-b}
 \sup_{s\in S}\,\|\mu_s\|<\infty.
\end{equation}
Furthermore, since $\mu_s\to\mu_0$ vaguely, Lemma~\ref{lemma-semi} with $F:=\mathbb R^n$ and $g:=1$ gives
\begin{equation}\label{11'}
 \mu_0(\mathbb R^n)\leqslant\liminf_{s\in S}\,\mu_s(\mathbb R^n)=1,
\end{equation}
whereas \cite[Section~IV.4.4, Corollary~3]{B2} yields
\begin{equation}\label{upper}\int1_K\,d\mu_0\geqslant\limsup_{s\in S}\,\int1_K\,d\mu_s\text{ \ for every compact $K\subset\mathbb R^n$},\end{equation}
the indicator function $1_K$ being bounded, of compact support, and upper semicontinuous on $\mathbb R^n$. Combining (\ref{11'}) and (\ref{upper}) with
\[\mu_0(\mathbb R^n)=\lim_{K\uparrow\mathbb R^n}\,\int1_K\,d\mu_0,\]
we get
\begin{equation}\label{110'}1\geqslant\mu_0(\mathbb R^n)\geqslant\limsup_{(s,K)\in S\times\mathfrak C}\,\int1_K\,d\mu_s=
1-\liminf_{(s,K)\in S\times\mathfrak C}\,\int1_{A\setminus K}\,d\mu_s,\end{equation}
$\mathfrak C:=\mathfrak C_{\mathbb R^n}$ being the upward directed set of all compact subsets $K$ of $\mathbb R^n$, and $S\times\mathfrak C$ the directed product of the directed sets $S$ and
$\mathfrak C$, cf.\ \cite[p.~68]{K}. (Note that the equality in (\ref{110'}) is implied by the fact that each $\mu_s$ is a positive measure of unit total mass concentrated on $A$.) The proof of (\ref{exm}) is thus reduced to that of
\begin{equation}\label{0}
  \liminf_{(s,K)\in S\times\mathfrak C}\,\int1_{A\setminus K}\,d\mu_s=0.
\end{equation}

By \cite[Theorem~2.6]{L} applied to $A\setminus K$, $K\in\mathfrak C$ being arbitrarily chosen, there exists the (unique) inner capacitary measure $\gamma_{A\setminus K}$, minimizing the energy $\|\cdot\|^2$ over the (convex) class $\Gamma_{A\setminus K}$ of all $\mu\in\mathcal E^+$ with the property
\[U^\mu\geqslant1\text{ \  n.e.\ on $A\setminus K$.}\]
For any $K'\in\mathfrak C$ such that $K\subset K'$, we have $\Gamma_{A\setminus K}\subset\Gamma_{A\setminus K'}$, and  \cite[Lemma~2.2]{L} therefore gives
\begin{equation}\label{str-ca}\|\gamma_{A\setminus K}-\gamma_{A\setminus K'}\|^2\leqslant\|\gamma_{A\setminus K}\|^2-\|\gamma_{A\setminus K'}\|^2.\end{equation}
Since $\|\gamma_{A\setminus K}\|^2=c_*(A\setminus K)$ \cite[Theorem~2.6]{L}, $\|\gamma_{A\setminus K}\|^2$ decreases as $K$ ranges through $\mathfrak C$, which together with (\ref{str-ca}) implies that the net $(\gamma_{A\setminus K})_{K\in\mathfrak C}\subset\mathcal E^+$ is Cauchy in the strong topology on $\mathcal E^+$. Noting that $(\gamma_{A\setminus K})_{K\in\mathfrak C}$ converges vaguely to zero,\footnote{Indeed, for any given $\varphi\in C_0(\mathbb R^n)$, there exists a relatively compact, open set $G\subset\mathbb R^n$ such that $\varphi(x)=0$ for all $x\not\in\overline{G}$. Hence, $\gamma_{A\setminus K}(\varphi)=0$ for all $K\in\mathfrak C$ with $K\supset\overline{G}$, and the claim follows.} we get
\begin{equation}\label{str-conv}
  \gamma_{A\setminus K}\to0\text{ \ strongly in $\mathcal E^+$ as $K\uparrow\mathbb R^n$,}
\end{equation}
the $\alpha$-Riesz kernel being perfect (Section~\ref{sec-gen}).

It follows from the above that
\begin{equation}\label{ii}U^{\gamma_{A\setminus K}}\geqslant1_{A\setminus K}\text{ \ n.e.\ on $A\setminus K$},\end{equation}
hence $\mu_s$-a.e.\ for all $s\in S$, the latter being derived from Lemma~\ref{l-negl} (see below) due to the fact that $A\setminus K$ along with $A$ is $\mu_s$-mea\-sur\-ab\-le, and therefore so is the set $E$ of all $x\in A\setminus K$ having the property $U^{\gamma_{A\setminus K}}(x)<1$. Integrating (\ref{ii}) with respect to $\mu_s$ we thus obtain, by the Cauchy--Schwarz (Bunyakovski) inequality,
\[\int1_{A\setminus K}\,d\mu_s\leqslant\int U^{\gamma_{A\setminus K}}\,d\mu_s\leqslant\|\gamma_{A\setminus K}\|\cdot\|\mu_s\|\text{ \ for all $K\in\mathfrak C$ and $s\in S$},\]
which combined with (\ref{str-b}) and (\ref{str-conv}) establishes (\ref{0}), whence (\ref{exm}).
\end{proof}

\begin{definition}[{\rm Landkof \cite[Section~II.1.2]{L}}]
  A measure $\mu\in\mathfrak M$ is said to be {\it $C$-ab\-s\-o\-lutely continuous} if $\mu(K)=0$ for any compact set $K\subset\mathbb R^n$ of zero capacity. Every $\mu\in\mathcal E$ is $C$-absolutely continuous, but not conversely.
\end{definition}

\begin{lemma}\label{l-negl}Given a $C$-ab\-s\-o\-lutely continuous measure $\mu\in\mathfrak M^+$, let $E\subset\mathbb R^n$ be a $\mu$-mea\-sur\-ab\-le set with $c_*(E)=0$. Then $E$ is $\mu$-neg\-lig\-ible, that is, $\mu^*(E)=0$.\end{lemma}

\begin{proof}
   Since $\mathbb R^n$ is representable as a countable union of compact sets, it is enough to show that $E$ is locally $\mu$-negligible, or equivalently, that $\mu_*(E)=0$. In view of the $C$-ab\-s\-o\-lute continuity of $\mu$, this follows immediately from $c_*(E)=0$.
\end{proof}

\begin{theorem}\label{quasi-complete}
Given $A\subset\mathbb R^n$ with $c_*(A)<\infty$, assume $(\mathcal P_4)$ is fulfilled. Then $\breve{\mathcal E}^+(A)$ along with $\mathcal E^+(A)$ is strongly closed, and hence strongly complete.
\end{theorem}

\begin{proof}
Since according to $(\mathcal P_4)$, $\mathcal E^+(A)$ is strongly closed, applying Lemma~\ref{str-cl-1} shows that so is $\breve{\mathcal E}^+(A)$. It remains to observe that any strongly closed subset of the strongly complete cone $\mathcal E^+$ must be strongly complete as well.
\end{proof}

\subsection{Continuity properties of the $f$-weighted energy}\label{sec-some}
Recall that an external field $f$ is assumed to be of form (\ref{f1}), where $\psi$, $\vartheta$, and $\omega$ satisfy $(\mathcal P_1)$--$(\mathcal P_3)$. Then the $f$-weighted energy $I_f(\mu)$ along with $\int f\,d\mu$ is well defined as a finite number or $+\infty$ for all bounded $\mu\in\mathcal E^+(\overline{A})$, see Lemma~\ref{int}.

\begin{lemma}\label{l-str-sem} For any $A\subset\mathbb R^n$ with $c_*(A)<\infty$, assume a net $(\mu_s)_{s\in S}\subset\breve{\mathcal E}^+(A)$ converges strongly to $\mu_0\in\mathcal E^+$. Then
\begin{equation*}\label{en-conv}
 I_f(\mu_0)\leqslant\liminf_{s\in S}\,I_f(\mu_s).
\end{equation*}
\end{lemma}

\begin{proof} Since the energy norm is strongly continuous on $\mathcal E$, we only need to verify that
\begin{equation}\label{lstr}
 \int f\,d\mu_0\leqslant\liminf_{s\in S}\,\int f\,d\mu_s.
\end{equation}

Applying the Cauchy--Schwarz inequality to the (signed) measures $\vartheta$ and $\mu_s-\mu_0$, elements of the pre-Hil\-bert space $\mathcal E$, we get
\[\left|\langle\vartheta,\mu_s-\mu_0\rangle\right|\leqslant\|\vartheta\|\cdot\|\mu_s-\mu_0\|,\]
which yields, by the strong convergence of $(\mu_s)_{s\in S}$ to $\mu_0$,
\begin{equation}\label{theta}
\langle\vartheta,\mu_0\rangle=\lim_{s\in S}\,\langle\vartheta,\mu_s\rangle,
\end{equation}
whence (\ref{lstr}) for $U^\vartheta$ in place of $f$.

Since $c_*(A)<\infty$ while $(\mu_s)_{s\in S}\subset\breve{\mathcal E}^+(A)$ converges strongly (hence vaguely) to $\mu_0$,  Lemma~\ref{str-cl-1} gives $\mu_0(\mathbb R^n)=1$, whence (\ref{mu}). Therefore, applying Lemma~\ref{lemma-semi} with $F:=\overline{A}$ to each of the functions $\psi$, $U^{\omega^+}$, and $-U^{\omega^-}$, and then combining the inequalities obtained with (\ref{theta}) we get (\ref{lstr}). (For $\psi$, we use the fact that a l.s.c.\ function on a compact set is lower bounded. For $-U^{\omega^-}$, we observe from $(\mathcal P_3)$ that its restriction to $\overline{A}$ is continuous (hence l.s.c.) and lower bounded.)
\end{proof}

\subsection{Quasiclosed sets}\label{secqu} For any $A\subset\mathbb R^n$, let $\mathcal E'(A)$ stand for the closure of $\mathcal E^+(A)$ in the strong topology on $\mathcal E^+$. Being a strongly closed subcone of the strongly complete cone $\mathcal E^+$, $\mathcal E'(A)$ is likewise complete in the (induced) strong topology.

We shall be mainly interested in the case when $(\mathcal P_4)$ is fulfilled (see Section~\ref{sec-intr2}), or equivalently, when
\begin{equation*}\label{Ecl}
 \mathcal E'(A)=\mathcal E^+(A).
\end{equation*}
A sufficient condition for this to occur is given in Theorem~\ref{l-quasi} below.
(In general,
\begin{equation}\label{belong}
\mathcal E'(A)\subset\mathcal E^+(\overline{A}),
\end{equation}
$\mathcal E^+(\overline{A})$ being strongly closed according to Theorem~\ref{l-quasi}.)

\begin{definition}[{\rm B.~Fuglede \cite{F71}}] A set $A\subset\mathbb R^n$ is said to be {\it quasiclosed} if
\begin{equation*}\label{def-q}
\inf\,\bigl\{c^*(A\bigtriangleup F):\ F\text{ closed, }F\subset\mathbb R^n\bigr\}=0,
\end{equation*}
$\bigtriangleup$ being the symmetric difference. Replacing here "closed" by "compact", we arrive at the concept of {\it quasicompact} sets.
\end{definition}

\begin{theorem}\label{l-quasi}If $A\subset\mathbb R^n$ is quasiclosed {\rm(}or in particular quasicompact{\rm)}, then the cone $\mathcal E^+(A)$ is strongly closed.
\end{theorem}

\begin{proof}
Given a sequence $(\mu_j)\subset\mathcal E^+(A)$ converging strongly (hence vaguely) to $\mu_0\in\mathcal E^+$, we only need to show that $\mu_0$ is concentrated on $A$. For closed $A$, the cone $\mathfrak M^+(A)$ is vaguely closed according to \cite[Section~III.2, Proposition~6]{B2}, whence $\mu_0\in\mathfrak M^+(A)$.

For quasiclosed $A$, note that for every $q\in(0,\infty)$, $\mathcal E^+_q:=\{\mu\in\mathcal E^+:\ \|\mu\|\leqslant q\}$
is hereditary \cite[Definition~5.2]{Fu4} and vaguely compact \cite[Lemma~2.5.1]{F1}, the Riesz kernels being strictly pseudo-definite, cf.\ \cite[p.~150]{F1}.
Since a strongly convergent sequence is certainly strongly bounded, $(\mu_j)\subset\mathcal E^+_{q'}$ for some $q'\in(0,\infty)$. Applying \cite[Corollary~6.2]{Fu4} with $\mathcal J:=\mathcal E^+_{q'}$ and $H:=A$, we thus derive that $\widetilde{\mathcal J}:=\mathcal E^+_{q'}\cap\mathfrak M^+(A)$ is vaguely compact, and, therefore, $(\mu_j)$ $({}\subset\widetilde{\mathcal J})$ has a vague limit point $\nu_0\in\widetilde{\mathcal J}$. The vague topology being Hausdorff, $\nu_0=\mu_0$, whence $\mu_0\in\widetilde{\mathcal J}\subset\mathfrak M^+(A)$ as desired.\end{proof}

\section{Proofs of the main results}\label{sec-proofs}

$\bullet$ Throughout Sections~\ref{sec-lemmas}--\ref{sec-pr2c2}, $\alpha\in(0,n)$ is {\it arbitrary}. Furthermore, an external field $f$ is of form (\ref{f1}), where $\psi$, $\vartheta$, and $\omega$ satisfy $(\mathcal P_1)$--$(\mathcal P_3)$ as well as (\ref{iv;}).

\subsection{Extremal measures}\label{sec-lemmas} A net $(\mu_s)_{s\in S}\subset\breve{\mathcal E}^+_f(A)$ is said to be {\it minimizing} (in Problem~\ref{pr-main}) if
\begin{equation}\label{min}
\lim_{s\in S}\,I_f(\mu_s)=w_f(A).
\end{equation}
We denote by $\mathbb M_f(A)$ the (nonempty) set of all those nets $(\mu_s)_{s\in S}$.

\begin{lemma}\label{la1}
  There is the unique $\xi_{A,f}\in\mathcal E^+$ such that, for each $(\mu_s)_{s\in S}\in\mathbb M_f(A)$,
  \begin{equation}\label{conv}
    \mu_s\to\xi_{A,f}\text{ \ strongly and vaguely in $\mathcal E^+$}.
  \end{equation}
This $\xi_{A,f}$ is said to be the extremal measure {\rm(}in Problem~{\rm\ref{pr-main}}{\rm)}.
\end{lemma}

\begin{proof} We shall first show that for any $(\mu_s)_{s\in S}$ and $(\nu_t)_{t\in T}$
from $\mathbb M_f(A)$,
\begin{equation}\label{ST}
\lim_{(s,t)\in S\times T}\,\|\mu_s-\nu_t\|=0,
\end{equation}
$S\times T$ being the directed product of the directed sets $S$ and
$T$. In fact, due to the convexity of the class $\breve{\mathcal E}^+_f(A)$, for any $(s,t)\in S\times T$ we have
\[4w_f(A)\leqslant4I_f\Bigl(\frac{\mu_s+\nu_t}{2}\Bigr)=\|\mu_s+\nu_t\|^2+4\int f\,d(\mu_s+\nu_t).\]
Applying the parallelogram identity in the pre-Hilbert space $\mathcal E$ therefore gives
\[0\leqslant\|\mu_s-\nu_t\|^2\leqslant-4w_f(A)+2I_f(\mu_s)+2I_f(\nu_t),\]
which together with (\ref{fin}) and (\ref{min}) yields (\ref{ST}) by letting $(s,t)$ range through $S\times T$.

Taking the two nets in (\ref{ST}) to be equal, we infer that every $(\nu_t)_{t\in T}\in\mathbb M_f(A)$ is strong Cauchy. The cone $\mathcal E^+$ being strongly complete, this $(\nu_t)_{t\in T}$ converges strongly to some $\xi_{A,f}\in\mathcal E^+$. The same (unique) $\xi_{A,f}$ also serves as the strong limit of any other $(\mu_s)_{s\in S}\in\mathbb M_f(A)$, which is obvious from (\ref{ST}). The strong topology on $\mathcal E^+$ being finer than the vague topology on $\mathcal E^+$, $(\mu_s)_{s\in S}$ must converge to $\xi_{A,f}$ also vaguely.
\end{proof}

\begin{remark}\label{tmass}In general, the extremal measure $\xi_{A,f}$ is {\it not} concentrated on $A$. What is clear so far is that
\begin{equation*}\label{Eq}
\xi_{A,f}\in\mathcal E'(A)\subset\mathcal E^+(\overline{A}),
\end{equation*}
the former relation being obvious from (\ref{conv}), and the latter from (\ref{belong}).\end{remark}

\begin{remark}\label{tmass'}Another consequence of (\ref{conv}) is that
\begin{equation}\label{11}
\xi_{A,f}(\mathbb R^n)\leqslant1,
\end{equation}
$\mu\mapsto\mu(\mathbb R^n)$ being vaguely l.s.c.\ on $\mathfrak M^+$ (Lemma~\ref{lemma-semi} with $F:=\mathbb R^n$ and $g:=1$). {\it Equality necessarily prevails in {\rm(\ref{11})} if $A$ is compact} \cite[Section~III.1.9, Corollary~3]{B2}.
\end{remark}

\begin{corollary}\label{cor-sol-extr}
  If the solution $\lambda_{A,f}$ to Problem~{\rm\ref{pr-main}} exists, then necessarily
  \begin{equation}\label{111}\lambda_{A,f}=\xi_{A,f}.\end{equation}
\end{corollary}

\begin{proof}
The trivial sequence $(\lambda_{A,f})$ being obviously minimizing:
\[(\lambda_{A,f})\in\mathbb M_f(A),\]
it must converge strongly in $\mathcal E^+$ to the extremal measure $\xi_{A,f}$ (Lemma~\ref{la1}), as well as to $\lambda_{A,f}$. Since the strong topology on $\mathcal E$ is Hausdorff, (\ref{111}) follows.
\end{proof}

\subsection{Proof of Theorem~\ref{th2'}}\label{sec-pr1} Since under the assumptions of the theorem, $\mathcal E^+(A)$ is strongly closed, see $(\mathcal P_4)$, so must be $\breve{\mathcal E}^+(A)$ (Theorem~\ref{quasi-complete}). Fix a minimizing net $(\mu_s)\in\mathbb M_f(A)$; according to Lemma~\ref{la1}, it converges strongly and vaguely to the extremal measure $\xi_{A,f}$. As $(\mu_s)\subset\breve{\mathcal E}^+_f(A)$, the strong closedness of $\breve{\mathcal E}^+(A)$ yields
\begin{equation}\label{exm'}
\xi_{A,f}\in\breve{\mathcal E}^+(A).
\end{equation}
Applying Lemma~\ref{l-str-sem} we further obtain
\begin{equation}\label{en-conv-xi}
 I_f(\xi_{A,f})\leqslant\lim_{s\in S}\,I_f(\mu_s)=w_f(A),
\end{equation}
the equality being valid by virtue of (\ref{min}). As $I_f(\xi_{A,f})>-\infty$ (Lemma~\ref{int}), we infer from (\ref{fin}), (\ref{exm'}), and (\ref{en-conv-xi}) that, actually,
\begin{equation}\label{incl}
\xi_{A,f}\in\breve{\mathcal E}_f^+(A),
\end{equation}
whence $I_f(\xi_{A,f})\geqslant w_f(A)$, which combined with (\ref{en-conv-xi}) gives
\[I_f(\xi_{A,f})=w_f(A).\]
This together with (\ref{incl}) shows that the extremal measure $\xi_{A,f}$ serves as the solution $\lambda_{A,f}$ to Problem~\ref{pr-main}, thereby completing the proof of the theorem.

\subsection{Proof of Theorem~\ref{th4'}}\label{sec-pr2} We first show that for any $g\in\Phi(\overline{A})$ and $\mu\in\mathfrak M^+(A)$,
\begin{equation}\label{limg}
 \int g\,d\mu=\lim_{K\uparrow A}\,\int g\,d\mu|_K,
\end{equation}
$\mu|_K$ being the trace of $\mu$ to $K$. Indeed, if $g\geqslant0$, this is implied by \cite[Lemma~1.2.2]{F1}, noting that for  $\mu\in\mathfrak M^+(A)$, the set $A$ is $\mu$-measurable and, moreover, $\mu=\mu|_A$. Otherwise, $\overline{A}$ must be compact, and the claim follows by applying (\ref{limg}) to a function $g':=g+c\geqslant0$ on $\overline{A}$, where $c\in(0,\infty)$, and then by making use of the fact that
\[\lim_{K\uparrow A}\,\mu(K)=\mu(A)<\infty,\]
the measures on (compact) $\overline{A}$ being bounded.

For any $\mu\in\breve{\mathcal E}_f^+(A)$, $\mu(K)\uparrow1$ as $K\uparrow A$. Applying (\ref{limg}) to each of the $\mu$-int\-egr\-able functions $\kappa_\alpha$, $\psi$, $U^{\vartheta^+}$, $U^{\vartheta^-}$, $U^{\omega^+}$, and $U^{\omega^-}$ therefore gives
\begin{equation*}
I_f(\mu)=\lim_{K\uparrow A}\,I_f(\mu|_K)=\lim_{K\uparrow A}\,I_f(\nu_K)\geqslant\lim_{K\uparrow A}\,w_f(K),
\end{equation*}
where $\nu_K:=\mu|_K/\mu(K)\in\breve{\mathcal E}_f^+(K)$, $K\geqslant K_0$. Letting $\mu$ range over $\breve{\mathcal E}_f^+(A)$ we thus get
\[w_f(A)\geqslant\lim_{K\uparrow A}\,w_f(K),\]
whence
\begin{equation}\label{limwe}
w_f(K)\downarrow w_f(A)\text{ \ as $K\uparrow A$},
\end{equation}
the net $\bigl(w_f(K)\bigr)_{K\in\mathfrak C_A}$ being decreasing and bounded from below by $w_f(A)$.

In view of (\ref{fin}), there exists, therefore, $K_0\in\mathfrak C_A$ such that, for each $K\geqslant K_0$, $w_f(K)$ is finite, and hence, according to Corollary~\ref{qcomp}, Problem~\ref{pr-main} with $A:=K$ is solvable. Noting from (\ref{limwe}) that those solutions $\lambda_{K,f}$ form a minimizing net:
\[(\lambda_{K,f})_{K\geqslant K_0}\in\mathbb M_f(A),\]
we obtain, by virtue of Lemma~\ref{la1},
\begin{equation}\label{eq-conv}
 \lambda_{K,f}\to\xi_{A,f}\text{ \ strongly and vaguely as $K\uparrow A$}.
\end{equation}
Since $\lambda_{A,f}$ exists by assumption, $\xi_{A,f}=\lambda_{A,f}$ according to Corollary~\ref{cor-sol-extr}, which substituted into (\ref{eq-conv}) proves (\ref{conv2}).

Thus, by (\ref{conv2}),
\[\lim_{K\uparrow A}\,\|\lambda_{K,f}\|^2=\|\lambda_{A,f}\|^2.\]
On the other hand, (\ref{limwe}) can be rewritten in the form
\[\lim_{K\uparrow A}\,\Bigl(\|\lambda_{K,f}\|^2+2\int f\,d\lambda_{K,f}\Bigr)=\|\lambda_{A,f}\|^2+2\int f\,d\lambda_{A,f}.\]
The last two relations combined imply (\ref{conv3}), cf.\ (\ref{cc}), thereby completing the proof.

\subsection{Proof of Theorem~\ref{th4'c}}\label{sec-pr2c}
Let $A$ be the union of an increasing sequence of universally measurable sets $A_j$.
For any l.s.c.\ $g:\overline{A}\to[0,\infty]$ and any $\mu\in\mathfrak M^+(A)$,
\begin{equation}\label{limg1}
 \int g\,d\mu=\lim_{j\to\infty}\,\int1_{A_j}g\,d\mu=\lim_{j\to\infty}\,\int g\,d\mu|_{A_j},
\end{equation}
where the former equality holds true by the monotone convergence theorem \cite[Section~IV.1, Theorem~3]{B2}
applied to the increasing sequence $(1_{A_j}g)$ of positive functions with the upper envelope $g$, and the latter by \cite{E2}, Propositions~4.14.1(b) and 4.14.6(3).

If now $g\in\Phi(\overline{A})$ and $g\ngeqslant0$, then $\overline{A}$ must be compact, and the proof of (\ref{limg1}) runs as usual --- namely, by replacing $g$ by $g':=g+c$, where $c\in(0,\infty)$, and then by making use of the fact that, due to (\ref{limg1}) with $g:=1$,
\[\lim_{j\to\infty}\,\mu(A_j)=\mu(A)<\infty.\]

Having thus established (\ref{limg1}) for any $g\in\Phi(\overline{A})$, we further arrive at (\ref{cont-count1}) by means of the same arguments as in the proof of (\ref{limwe}).

The rest of the proof runs similarly to that in Section~\ref{sec-pr2}.
To be precise, since the minimizers $\lambda_{A_j,f}$ and $\lambda_{A,f}$ exist by assumption, it follows from (\ref{cont-count1}) that $\lambda_{A_j,f}$ form a minimizing sequence,
which, according to Lemma~\ref{la1} and Corollary~\ref{cor-sol-extr}, must converge to $\lambda_{A,f}$ strongly and vaguely. This establishes (\ref{cont-count12}). Finally, (\ref{cont-count1}) and (\ref{cont-count12}) result in (\ref{cont-count13}) in exactly the same way as in the last paragraph of Section~\ref{sec-pr2}.

\subsection{Proof of Theorem~\ref{th4'c2}}\label{sec-pr2c2} We begin by noting that $A$, being a countable intersection of quasiclosed sets $A_j$, is likewise quasiclosed \cite[Lemma~2.3]{F71}. Since $(\mathcal P_4)$ is fulfilled for quasiclosed sets (Theorem~\ref{l-quasi}), applying Theorem~\ref{th2'} shows that the minimizers $\lambda_{A_j,f}$, $j\geqslant j_0$, as well as $\lambda_{A,f}$ do exist.

By the convexity of $\breve{\mathcal E}^+_f(A_j)$, $(\lambda_{A_j,f}+\lambda_{A_k,f})/2\in\breve{\mathcal E}^+_f(A_j)$ for all  $k\geqslant j$, hence
\[4w_f(A_j)\leqslant4I_f\biggl(\frac{\lambda_{A_j,f}+\lambda_{A_k,f}}{2}\biggr)=\|\lambda_{A_j,f}+\lambda_{A_k,f}\|^2+4\int f\,d(\lambda_{A_j,f}+\lambda_{A_k,f}).\]
Using the parallelogram identity in the pre-Hilbert space $\mathcal E$ therefore gives
\begin{equation}\label{fund}
\|\lambda_{A_j,f}-\lambda_{A_k,f}\|^2\leqslant-4w_f(A_j)+2I_f(\lambda_{A_j,f})+2I_f(\lambda_{A_k,f})=2w_f(A_k)-2w_f(A_j).
\end{equation}
The sequence $\bigl(w_f(A_k)\bigr)$ being increasing and bounded from above by $w_f(A)\in\mathbb R$, we infer from (\ref{fund}) that $(\lambda_{A_k,f})_{k\geqslant j}\subset\breve{\mathcal E}_f^+(A_j)$ is strong Cauchy, and hence it converges strongly and vaguely to some $\lambda_0\in\mathcal E^+$. We claim that
$\lambda_0=\lambda_{A,f}$.

Since $A_j$ is quasiclosed and of finite inner capacity, $\breve{\mathcal E}^+(A_j)$ is strongly closed (Theorems~\ref{quasi-complete} and \ref{l-quasi}), and therefore $\lambda_0\in\breve{\mathcal E}^+(A_j)$ for all $j\geqslant j_0$. Being thus a countable union of $\lambda_0$-neg\-lig\-ible sets $A_j^c$, the set $A^c$ is likewise $\lambda_0$-neg\-lig\-ible, whence
\begin{equation}\label{funddd}\lambda_0\in\breve{\mathcal E}^+(A).\end{equation}
Furthermore,
\begin{equation}\label{fundd}
-\infty<I_f(\lambda_0)\leqslant\liminf_{j\to\infty}\,I_f(\lambda_{A_j,f})=\lim_{j\to\infty}\,w_f(A_j)\leqslant w_f(A)<\infty,
\end{equation}
where the first and the second inequalities hold true by virtue of Lemmas~\ref{int} and \ref{l-str-sem}, respectively. In view of (\ref{funddd}) and (\ref{fundd}),
\begin{equation}\label{www}
 \lambda_0\in\breve{\mathcal E}_f^+(A),
\end{equation}
whence $I_f(\lambda_0)\geqslant w_f(A)$, which substituted into (\ref{fundd}) shows that, actually,
\begin{equation}\label{j}
 \lim_{j\to\infty}\,w_f(A_j)=w_f(A)=I_f(\lambda_0).
\end{equation}
As seen from (\ref{www}) and (\ref{j}), the measure $\lambda_0$, the strong and the vague limit of the sequence $(\lambda_{A_j,f})_{j\geqslant j_0}$, serves, indeed, as $\lambda_{A,f}$. To complete the proof, it remains to verify that $c_{A_j,f}\to c_{A,f}$ as $j\to\infty$, but this follows from the above in exactly the same manner as in the last paragraph of Section~\ref{sec-pr2}.

\subsection{The case of $\alpha\in(0,2]$. Auxiliary results}\label{sec-lemmas2} In the remainder of the paper,
\[\alpha\in(0,2].\]

Among the variety of equivalent definitions of inner $\alpha$-Riesz balayage
(see \cite{Z-bal,Z-bal2}, cf.\ also \cite{Z-arx1}--\cite{Z-arx}), here we have chosen the following one to start with.

\begin{definition}\label{def-bal}
  For any $\mu\in\mathfrak M^+$ and any $A\subset\mathbb R^n$, {\it the inner $\alpha$-Riesz balayage} $\mu^A\in\mathfrak M^+$ is the vague limit of the sequence $(\mu_j^A)\subset\mathcal E^+$, where $(\mu_j)\subset\mathcal E^+$ is such that
  \begin{equation}\label{up}
  U^{\mu_j}\uparrow U^\mu\text{ \ pointwise on $\mathbb R^n$ as $j\to\infty$},
  \end{equation}
while $\mu_j^A$ denotes the only measure in $\mathcal E'(A)$, the strong closure of $\mathcal E^+(A)$ (Section~\ref{secqu}), having the property $U^{\mu_j^A}=U^{\mu_j}$ n.e.\ on $A$.\footnote{A sequence $(\mu_j)\subset\mathcal E^+$ satisfying (\ref{up}) does exist (see \cite[p.~272]{L}, cf.\ \cite[p.~257]{Ca2}). Furthermore, for each $\chi\in\mathcal E^+$, there exists the unique measure $\chi^A\in\mathcal E'(A)$ such that $U^{\chi^A}=U^\chi$ n.e.\ on $A$ \cite[Theorem~3.1(c)]{Z-arx-22}; it can alternatively be characterized as the orthogonal projection of $\chi$ in the pre-Hilbert space $\mathcal E$ onto the convex, strongly complete cone $\mathcal E'(A)\subset\mathcal E^+$ \cite[Theorem~3.1(b)]{Z-arx-22}. Regarding the concept of orthogonal projection in a pre-Hilbert space, see e.g.\ \cite[Theorem~1.12.3]{E2}.\label{F}}
\end{definition}

The inner balayage $\mu^A$ thus defined does exist, is unique, and it does not depend on the choice of the above sequence $(\mu_j)$, cf.\ \cite[Section~3.3]{Z-bal}. Furthermore,
\begin{align}
U^{\mu^A}&=U^\mu\text{ \ n.e.\ on $A$},\label{eq-bal}\\
U^{\mu^A}&\leqslant U^\mu\text{ \ on $\mathbb R^n$}.\notag
\end{align}
The same $\mu^A$ is uniquely characterized within $\mathfrak M^+$ by the symmetry relation
\[I(\mu^A,\chi)=I(\mu,\chi^A)\text{ \ for all $\chi\in\mathcal E^+$,}\]
where $\chi^A$ denotes the only measure in $\mathcal E'(A)$ with $U^{\chi^A}=U^{\chi}$ n.e.\ on $A$ (footnote~\ref{F}).

\begin{remark}
In general, $\mu^A$ is {\it not\/} concentrated on $A$ (unless, of course, $A$ is closed). It is also worth noting that, even for closed $A$, there may exist infinitely many $\nu\in\mathfrak M^+(A)$ with $U^\nu=U^\mu$ n.e.\ on $A$; hence, (\ref{eq-bal}) {\it cannot}, in general, serve as a characteristic property of balayage. Compare with Theorem~\ref{th-bal}(iii$_1$).
\end{remark}

$\bullet$ In all that follows, we assume that $\mathcal E^+(A)$ is strongly closed (or in particular that $A$ is quasiclosed), so that $\mathcal E^+(A)=\mathcal E'(A)$, and that $f$ is of form (\ref{f3'}), i.e.
\begin{equation*}
f=-U^\tau-U^\sigma=-U^\delta,
\end{equation*}
where $\tau\in\mathcal E^+$, $\sigma\in\mathfrak M^+$ is a bounded measure with $d(S_\sigma,A)>0$, and
$\delta:=\tau+\sigma$. That is, $(\mathcal P_4)$, $(\mathcal P_2')$, and $(\mathcal P_3')$ are required to hold (see the beginning of Section~\ref{subs-main2}).

$\P$ It is crucial to the analysis below that for these $A$ and $\delta$, the inner balayage $\delta^A$ is a measure of {\it finite} energy, {\it concentrated} on $A$ (see Theorem~\ref{th-bal} for details).

\begin{theorem}\label{th-bal}
For the above $A$ and $\delta$, the inner balayage $\delta^A$ can equivalently be determined by means of any one of the following assertions.
\begin{itemize}
\item[{\rm(i$_1$)}] $\delta^A$ is the only measure of minimum energy in the class
\begin{equation}\label{Ga}
 \Gamma_{A,\delta}:=\bigl\{\nu\in\mathfrak M^+:\ U^\nu\geqslant U^\delta\text{ \ n.e.\ on $A$}\bigr\}.
 \end{equation}
That is, $\delta^A\in\Gamma_{A,\delta}$ and
 \begin{equation*}\label{delta-alt}
  I(\delta^A)=\min_{\nu\in\Gamma_{A,\delta}}\,I(\nu).
  \end{equation*}
  \item[{\rm(ii$_1$)}] $\delta^A$ is the only measure of minimum potential in the class $\Gamma_{A,\delta}$, introduced by means of {\rm(\ref{Ga})}. That is, $\delta^A\in\Gamma_{A,\delta}$  and
  \begin{equation*}\label{Gap}U^{\delta^A}=\min_{\nu\in\Gamma_{A,\delta}}\,U^\nu\text{ \ on\/ $\mathbb R^n$}.\end{equation*}
  \item[{\rm(iii$_1$)}] $\delta^A$ is the only measure in the class $\mathcal E^+(A)$ having the property
  \[U^{\delta^A}=U^\delta\text{ \  n.e.\ on $A$.}\]
    \item[{\rm(iv$_1$)}] $\delta^A$ can alternatively be characterized by any one of the three limit relations
  \begin{gather}
  \label{char1}\delta^K\to\delta^A\text{ \ strongly in $\mathcal E^+$ as $K\uparrow A$},\\
   \label{char2}\delta^K\to\delta^A\text{ \ vaguely in $\mathfrak M^+$ as $K\uparrow A$},\\
   U^{\delta^K}\uparrow U^{\delta^A}\text{ \ pointwise on $\mathbb R^n$ as $K\uparrow A$},\notag
  \end{gather}
  where $\delta^K$ denotes the only measure in $\mathcal E^+(K)$ with $U^{\delta^K}=U^\delta$ n.e.\ on $K$.
\end{itemize}
\end{theorem}

 \begin{proof} To simplify notations, for $\mu\in\mathfrak M^+$ write $\mu':=\mu^{\overline{A}}$.
  We begin by showing that
  \begin{equation}\label{del-f}
 \delta'\in\mathcal E^+(\overline{A}).
\end{equation}
Since obviously $\tau'\in\mathcal E^+(\overline{A})$, this reduces to $\sigma'\in\mathcal E^+(\overline{A})$. As $S_\sigma\cap\overline{A}=\varnothing$, $\sigma'$ is $C$-ab\-s\-ol\-utely continuous (see \cite[Corollary~3.19]{FZ} or, more generally, \cite[Corollary~5.2]{Z-bal2}), and therefore
 $U^{\sigma'}=U^\sigma$ $\sigma'$-a.e.\ on $\mathbb R^n$, cf.\ (\ref{eq-bal}).
 This implies by integration\footnote{Here we use the fact that for any $\mu\in\mathfrak M^+$ and any $Q\subset\mathbb R^n$,  $\mu^Q(\mathbb R^n)\leqslant\mu(\mathbb R^n)$ \cite[Corollary~4.9]{Z-bal}.\label{f-cor}}
 \[\int U^{\sigma'}\,d\sigma'=\int U^\sigma\,d\sigma'\leqslant d(S_{\sigma},A)^{\alpha-n}\sigma(\mathbb R^n)^2<\infty,\]
 cf.\ (\ref{f3''}), whence the claim.

 But, by \cite[Corollary~4.2]{Z-bal} ("balayage with a rest"),\footnote{See also Landkof's book \cite[p.~264]{L}, where the sets in question were assumed to be Borel, while the measures bounded.}
 \[\delta^{A}=(\delta')^{A}.\]
 Therefore, applying \cite[Theorem~3.1(c)]{Z-arx-22} to $\delta'\in\mathcal E^+$, cf.\ (\ref{del-f}), we obtain (iii$_1$) by making use of the equality $\mathcal E^+(A)=\mathcal E'(A)$ as well as of the fact that for any $\nu\in\mathfrak M^+$,
 \begin{equation*}\label{impl}
 U^\nu=U^{\delta'}\text{ \ n.e.\ on $A$}\Longrightarrow U^\nu=U^\delta\text{ \ n.e.\ on $A$},
 \end{equation*}
 which in turn is implied by (\ref{eq-bal}) (applied to $\overline{A}$ and $\delta$) and the strengthened version of countable subadditivity for inner capacity (footnote~\ref{foot-sub}).

 The remaining assertions (i$_1$), (ii$_1$), and (iv$_1$) follow from \cite{Z-arx-22} (Definition~3.1 and Theorem~3.1, (a) and (d))
 in the same manner as above.
\end{proof}

\begin{lemma}\label{LL}
 For the above $A$ and $\delta$, we have
 \begin{equation}\label{ext1}
 \xi_{A,f}\in\mathcal E^+(A)\text{ \ and \ }I_f(\xi_{A,f})=w_f(A),
 \end{equation}
 $\xi_{A,f}$ being the extremal measure. Problem~{\rm\ref{pr-main}} is therefore solvable if and only if
\begin{equation}\label{ext2}\xi_{A,f}(\mathbb R^n)=1,\end{equation}
and in the affirmative case $\xi_{A,f}=\lambda_{A,f}$.
\end{lemma}

\begin{proof} Fix $(\mu_s)_{s\in S}\in\mathbb M_f(A)$. Since $\mu_s\to\xi_{A,f}$ strongly (Lemma~\ref{la1}), whereas $\mathcal E^+(A)$ is strongly closed, see $(\mathcal P_4)$, the former relation in (\ref{ext1}) is obvious.

Furthermore,
\begin{equation}\label{limit2}\lim_{s\in S}\,\int f\,d\mu_s=\int f\,d\xi_{A,f},\end{equation}
because
\begin{equation*}\label{limit2'}
\lim_{s\in S}\,\int U^\delta\,d\mu_s=\lim_{s\in S}\,\langle\delta^A,\mu_s\rangle=\langle\delta^A,\xi_{A,f}\rangle=\int U^\delta\,d\xi_{A,f}.
\end{equation*}
Indeed, the last (similarly, the first) equality holds true since $U^\delta=U^{\delta^A}$ n.e.\ on $A$, hence $\xi_{A,f}$-a.e.\ (see the former relation in (\ref{ext1}) as well as Lemma~\ref{l-negl}); whereas the second equality follows from (\ref{conv}) by making use of the Cauchy--Schwarz inequality, applied to $\delta^A$ and $\mu_s-\xi_{A,f}$, elements of the pre-Hil\-bert space $\mathcal E$ (cf.\ Theorem~\ref{th-bal}(iii$_1$)).

Adding (\ref{limit2}) multiplied by $2$ to
\begin{equation}\label{muconv}
\lim_{s\in S}\,\|\mu_s\|^2=\|\xi_{A,f}\|^2,
\end{equation}
cf.\ (\ref{conv}), on account of (\ref{min}) we get
\begin{equation*}\label{limit1}I_f(\xi_{A,f})=\lim_{s\in S}\,I_f(\mu_s)=w_f(A),\end{equation*}
whence the latter relation in (\ref{ext1}). Thus, actually, $\xi_{A,f}\in\mathcal E^+_f(A)$.

Now, assuming that (\ref{ext2}) is fulfilled, we derive from (\ref{ext1}) that $\xi_{A,f}$ must indeed serve as the minimizer $\lambda_{A,f}$. For the "only if" part, see Corollary~\ref{cor-sol-extr}.
\end{proof}

\begin{lemma}\label{LLL}
  For the extremal measure $\xi=\xi_{A,f}$, we have
  \begin{equation}\label{pot-extr}
   U^\xi_f\geqslant C_\xi\text{ \ n.e.\ on $A$,}\end{equation}
  where
  \begin{equation}\label{Xi}C_\xi:=\int U^\xi_f\,d\xi=w_f(A)-\int f\,d\xi\in(-\infty,\infty).\end{equation}
  If moreover $f$ is l.s.c.\ on $\overline{A}$, then also
  \begin{equation}
   U^\xi_f\leqslant C_\xi\text{ \ on $S(\xi)$,}\label{pot-extr2}
  \end{equation}
  and hence
  \begin{equation}
   U^\xi_f=C_\xi\text{ \ $\xi$-a.e.}\label{pot-extr2'}
  \end{equation}
  \end{lemma}

\begin{proof} By Corollary~\ref{qcomp}, there exists the (unique) solution $\lambda_{K,f}$ to Problem~\ref{pr-main} with $A:=K$, where $K\in\mathfrak C_A$ is large enough ($K\geqslant K_0$), while according to Theorem~\ref{th-ch1},
\begin{equation}\label{K}
U_f^{\lambda_{K,f}}\geqslant\int U^{\lambda_{K,f}}_f\,d\lambda_{K,f}\text{ \ n.e.\ on $K$.}
\end{equation}
Since $w_f(K)\downarrow w_f(A)$ as $K\uparrow A$, see (\ref{limwe}), these $\lambda_{K,f}$ form a minimizing net, i.e.
\[(\lambda_{K,f})_{K\geqslant K_0}\in\mathbb M_f(A),\]
and hence $(\lambda_{K,f})$ converges strongly in $\mathcal E^+$ to the extremal measure $\xi$ (Lemma~\ref{la1}). Relations  (\ref{limit2}) and (\ref{muconv}), both with $(\lambda_{K,f})_{K\in\mathfrak C_A}$ in place of $(\mu_s)_{s\in S}$, give
\begin{equation}\label{conv5}
\lim_{K\uparrow A}\,\int U^{\lambda_{K,f}}_f\,d\lambda_{K,f}=\int U^\xi_f\,d\xi=:C_\xi\in(-\infty,\infty),
\end{equation}
the finiteness of $C_\xi$ being clear from the latter relation in (\ref{ext1}).

Fix $K_*\in\mathfrak C_A$. The strong topology on $\mathcal E^+$ being first-countable, one can choose a subsequence $(\lambda_{K_j,f})_{j\in\mathbb N}$ of the net $(\lambda_{K,f})_{K\in\mathfrak C_A}$ such that
\begin{equation}\label{J}
 \lambda_{K_j,f}\to\xi\text{ \ strongly (hence vaguely) in $\mathcal E^+$ as $j\to\infty$.}
\end{equation}
There is certainly no loss of generality in assuming
$K_*\subset K_j$ for all $j$,
for if not, we replace $K_j$ by $K_j':=K_j\cup K_*$; then, by the monotonicity of $\bigl(w_f(K)\bigr)_{K\in\mathfrak C_A}$, the sequence $(\lambda_{K_j',f})_{j\in\mathbb N}$ remains minimizing, and hence also converges strongly to $\xi$.

Due to the arbitrary choice of $K_*\in\mathfrak C_A$, (\ref{pot-extr}) will follow once we show that
\begin{equation}\label{JJ'}U^\xi_f\geqslant C_\xi\text{ \ n.e.\ on $K_*$}.\end{equation}
Passing if necessary to a subsequence and changing the notation, we conclude from (\ref{J}), by making use of \cite[p.~166, Remark]{F1}, that
\begin{equation}\label{JJ}U^\xi=\lim_{j\to\infty}\,U^{\lambda_{K_j,f}}\text{ \ n.e.\ on $\mathbb R^n$}.\end{equation}
Therefore, applying (\ref{K}) to each $K_j$, and then letting $j\to\infty$, on account of (\ref{conv5}) and (\ref{JJ}) we arrive at (\ref{JJ'}). (Here the countable subadditivity of inner capacity on universally measurable sets has been utilized, see e.g.\ \cite[p.~144]{L}.)

Assume now that $f$ is l.s.c.\ on $\overline{A}$. By Theorem~\ref{th-ch1}, then
\[U_f^{\lambda_{K_j,f}}\leqslant\int U^{\lambda_{K_j,f}}_f\,d\lambda_{K_j,f}\text{ \ on\/ $S(\lambda_{K_j,f})$},\]
where $(K_j)\subset\mathfrak C_A$ is the sequence chosen above. Since $(\lambda_{K_j,f})$ converges to $\xi$ vaguely, see (\ref{J}), for every $x\in S(\xi)$ there exist a subsequence $(K_{j_k})$ of $(K_j)$ and points $x_{j_k}\in S(\lambda_{K_{j_k},f})$, $k\in\mathbb N$, such that $x_{j_k}$ approach $x$ as $k\to\infty$. Thus
\[U_f^{\lambda_{K_{j_k},f}}(x_{j_k})\leqslant\int U^{\lambda_{K_{j_k},f}}_f\,d\lambda_{K_{j_k},f}\text{ \ for all $k\in\mathbb N$}.\]
Letting here $k\to\infty$, in view of (\ref{conv5}) and the lower semicontinuity of the mapping $(x,\mu)\mapsto U^\mu(x)$ on $\mathbb R^n\times\mathfrak M^+$, where $\mathfrak M^+$ is equipped with the vague topology \cite[Lemma~2.2.1(b)]{F1}, we obtain (\ref{pot-extr2}). Finally, combining (\ref{pot-extr2}) with (\ref{pot-extr}) yields $U_f^\xi=C_\xi$ n.e.\ on $S(\xi)\cap A$, whence (\ref{pot-extr2'}), for $\xi\in\mathcal E^+(A)$ (Lemma~\ref{LL}).
\end{proof}

\subsection{Proof of Theorem~\ref{th3'}}\label{sec-pr3} Under the hypotheses $(\mathcal P_2')$, $(\mathcal P_3')$, and $(\mathcal P_4)$, assume moreover that (\ref{leq1}) is fulfilled, where $\delta$ is the measure appearing in (\ref{f3'}).

The proof is mainly based on Theorems~\ref{th-ch2} and \ref{th-bal}, and it is given in four steps.\smallskip

{\it Step 1.} Assume first that $c_*(A)<\infty$. Due to $(\mathcal P_2')$, $(\mathcal P_3')$, and $(\mathcal P_4)$, $\delta^A$ and $\gamma_A$, the inner balayage of $\delta$ to $A$ and the inner capacitary measure of $A$, respectively, are both of finite energy, and they are concentrated on $A$, i.e.
\begin{equation}\label{zk}\delta^A,\,\gamma_A\in\mathcal E^+(A)\end{equation}
(see Theorem~\ref{th-bal}(iii$_1$) and \cite[Theorem~7.2]{Z-arx-22}, respectively).\footnote{Also note that $\gamma_A$ is nonzero, which is obvious from $c_*(A)>0$, cf.\ (\ref{iv;}).} We aim to show that
\begin{equation}\label{lk}
\beta:=\delta^A+\eta_{A,f}\gamma_A,
\end{equation}
$\eta_{A,f}$ being introduced in (\ref{eqalt}),
serves as the (unique) solution to Problem~\ref{pr-main}. (Note that this would provide an alternative proof of the solvability of Problem~\ref{pr-main}; compare with Theorem~\ref{th2'} and its proof, given in Section~\ref{sec-pr1}.)

Indeed, \cite[Corollary~4.9]{Z-bal}, quoted in footnote~\ref{f-cor} above, yields
\begin{equation}\label{balbound}0\leqslant\delta^A(\mathbb R^n)\leqslant\delta(\mathbb R^n)\leqslant1,\end{equation}
the last inequality being valid by virtue of (\ref{leq1}). Thus $\eta_{A,f}\in[0,\infty)$, and combining $c_*(A)=\gamma_A(\mathbb R^n)$ (see \cite[Theorem~2.6]{L}) with (\ref{eqalt}), (\ref{zk}), and (\ref{lk}) shows that, actually, $\beta\in\breve{\mathcal E}^+(A)$. Moreover, $\beta\in\breve{\mathcal E}^+_f(A)$, for
\begin{equation*}\label{finn}
\int U^\delta\,d\beta=\langle\delta^A,\beta\rangle<\infty.
\end{equation*}
According to Theorem~\ref{th-ch2}, $\beta=\lambda_{A,f}$ will therefore follow once we verify the inequality
\begin{equation}\label{qk1}U_f^\beta\geqslant\int U_f^\beta\,d\beta\text{ \ n.e.\ on $A$}.\end{equation}

By the strengthened version of countable subadditivity for inner capacity, we infer from (\ref{lk}), $U^{\gamma_A}=1$ n.e.\ on $A$ \cite[p.~145]{L}, and (\ref{eq-bal}) (applied to $\delta$) that
\begin{equation}\label{qk2}U_f^\beta=U^\beta-U^\delta=U^{\delta^A-\delta}+\eta_{A,f}U^{\gamma_A}=\eta_{A,f}\text{ \ n.e.\ on $A$,}\end{equation}
hence $\beta$-a.e. Therefore,
\[\int U_f^\beta\,d\beta=\eta_{A,f}\beta(\mathbb R^n)=\eta_{A,f},\] which substituted into (\ref{qk2}) gives (\ref{qk1}). Thus the solution $\lambda_{A,f}$ to Problem~\ref{pr-main} does indeed exist, and moreover $\lambda_{A,f}=\beta$ and $c_{A,f}=\eta_{A,f}$,
$c_{A,f}$ being the inner $f$-weighted equilibrium constant. This establishes (\ref{const-alt}) as well as the representation
\begin{equation}\label{pres}
\lambda_{A,f}=\delta^A+\eta_{A,f}\gamma_A.
\end{equation}

As $(\gamma_A)^A=\gamma_A$ \cite[Lemma~9.1]{Z-arx-22}, identity (\ref{pres}) can be rewritten in the form
\begin{equation*}\label{bal}
\lambda_{A,f}=(\delta+\eta_{A,f}\gamma_A)^A.
\end{equation*}
Applying Theorem~\ref{th-bal}, (i$_1$) and (ii$_1$), to $\delta+\eta_{A,f}\gamma_A$ in place of $\delta$, which is possible because $\eta_{A,f}\gamma_A\in\mathcal E^+$, we therefore conclude that $\lambda_{A,f}$ can be characterized as the unique measure of minimum energy, resp.\ of minimum potential, within the class of all $\nu\in\mathfrak M^+$ having the property $U^\nu\geqslant U^\delta+\eta_{A,f}U^{\gamma_A}$  n.e.\ on $A$,
or equivalently
\[U^\nu_f\geqslant\eta_{A,f}\text{ \ n.e.\ on $A$}.\]
This establishes assertion (i), resp.\ (ii), of Theorem~\ref{th3'}.

Similarly, Theorem~\ref{th-bal}(iii$_1$) applied to $\delta+\eta_{A,f}\gamma_A$ results in Theorem~\ref{th3'}(iii).\smallskip

{\it Step 2.} Assume now that $c_*(A)=\infty$, and that condition (\ref{bal1}) is fulfilled; then necessarily $\delta^A(\mathbb R^n)=1$. Actually, $\delta^A\in\breve{\mathcal E}^+(A)$, see Theorem~\ref{th-bal}(iii$_1$), whence
\begin{equation}\label{zetaA}
\delta^A\in\breve{\mathcal E}^+_f(A),
\end{equation}
for
\[\int U^\delta\,d\delta^A=\langle\delta^A,\delta^A\rangle<\infty.\]
We aim to show that $\delta^A$ serves as the (unique) solution to Problem~\ref{pr-main}, i.e.
\begin{equation}\label{lb}
\delta^A=\lambda_{A,f}.
\end{equation}

Noting that
\begin{equation}\label{417}
U_f^{\delta^A}=U^{\delta^A}-U^\delta=0\text{ \ n.e.\ on $A$},
\end{equation}
hence $\delta^A$-a.e., we get
\[\int U_f^{\delta^A}\,d\delta^A=0,\]
which substituted into (\ref{417}) gives
\[U_f^{\delta^A}=\int U_f^{\delta^A}\,d\delta^A\text{ \ n.e.\ on $A$}.\]
By Theorem~\ref{th-ch2}, this together with (\ref{zetaA}) yields (\ref{lb}) as well as $c_{A,f}=0$. Furthermore, observing from (\ref{eqalt}) that $\eta_{A,f}$ along with $c_{A,f}$ equals $0$, we arrive at (\ref{const-alt}).

We finally note that, due to the equalities $\lambda_{A,f}=\delta^A$ and $\eta_{A,f}=0$ thus obtained, assertions (i)--(iii) of Theorem~\ref{th3'} follow directly from Theorem~\ref{th-bal}, (i$_1$)--(iii$_1$).\smallskip

{\it Step 3.} The aim of this step is to show that assumption (\ref{bal1}) is not only sufficient, but also necessary for the existence of the solution $\lambda_{A,f}$. Assume to the contrary that $\lambda_{A,f}$ exists, but (\ref{bal1}) fails to hold; in view of (\ref{balbound}), then necessarily
\begin{equation}\label{ci'}
 c_*(A)=\infty\text{ \ and \ }\delta^A(\mathbb R^n)<1.
\end{equation}

According to Corollary~\ref{qcomp}, for each $K\in\mathfrak C_A$ large enough ($K\geqslant K_0$), there is the solution $\lambda_{K,f}$ to Problem~\ref{pr-main} with $A:=K$; and moreover the net $(\lambda_{K,f})_{K\geqslant K_0}$ converges strongly and vaguely to the extremal measure $\xi_{A,f}$, determined by Lemma~\ref{la1}. We claim that, due to the former relation in (\ref{ci'}),
\begin{equation}\label{eqq}
 \xi_{A,f}=\delta^A.
\end{equation}

By (\ref{eqalt}) and (\ref{pres}), both applied to $K\geqslant K_0$,
\begin{equation}\label{fol}
\lambda_{K,f}=\delta^K+\widetilde{\eta}_{K,f}\lambda_K,
\end{equation}
where $\lambda_K:=\gamma_K/c(K)$ is the solution to Problem~\ref{pr-main} with $A:=K$ and $f:=0$, and
\[\widetilde{\eta}_{K,f}:=1-\delta^K(\mathbb R^n).\]
But the net $(\widetilde{\eta}_{K,f})_{K\geqslant K_0}\subset\mathbb R$ is bounded since, in view of (\ref{balbound}) with $A:=K$,
\[0\leqslant\delta^K(\mathbb R^n)\leqslant\delta(\mathbb R^n)\leqslant1\text{ \ for all $K\in\mathfrak C_A$}.\]
Furthermore, by virtue of (\ref{char1}) and (\ref{char2}),
\[\delta^K\to\delta^A\text{ \ strongly (and vaguely) in $\mathcal E^+$ as $K\uparrow A$}.\]
Thus, if we show that
\begin{equation}\label{lconv}
\lambda_K\to0\text{ \ strongly in $\mathcal E^+$ as $K\uparrow A$},
\end{equation}
identity (\ref{eqq}) will follow from (\ref{fol}) by passing to the limit as $K\uparrow A$, and making use of the triangle inequality in the pre-Hil\-bert space $\mathcal E$.

It is seen from (\ref{limwe}) that the net $(\lambda_K)_{K\geqslant K_0}$ is minimizing in Problem~\ref{pr-main} with $f=0$, i.e.\ $(\lambda_K)_{K\geqslant K_0}\in\mathbb M(A)$. Applying Lemma~\ref{la1}, we therefore conclude that there exists the unique extremal measure $\xi_A$ in Problem~\ref{pr-main} with $f=0$, and moreover $\lambda_K\to\xi_A$ strongly in $\mathcal E^+$ as $K\uparrow A$. This yields
\[\|\xi_A\|^2=\lim_{K\uparrow A}\,\|\lambda_K\|^2=\lim_{K\uparrow A}\,w(K)=0,\]
the last equality being caused by $c_*(A)=\infty$. Since the $\alpha$-Riesz kernel is strictly positive definite, we thus have $\xi_A=0$, which proves (\ref{lconv}), whence (\ref{eqq}).

Since $\lambda_{A,f}$ is assumed to exist, Corollary~\ref{cor-sol-extr} together with (\ref{eqq}) gives
\[\lambda_{A,f}=\xi_{A,f}=\delta^A,\]
which however is impossible, for $\delta^A(\mathbb R^n)<1$ by (\ref{ci'}).\smallskip

{\it Step 4.} To complete the proof of the theorem, it remains to establish (\ref{conv3'}). Applying (\ref{eqalt}) and (\ref{const-alt}) to each $K\in\mathfrak C_A$ large enough ($K\geqslant K_0$), we get
\[c_{K,f}=\frac{1-\delta^K(\mathbb R^n)}{c(K)}.\]
In view of (\ref{conv3}), (\ref{conv3'}) will therefore follow once we show that the net $\bigl(\delta^K(\mathbb R^n)\bigr)_{K\in\mathfrak C_A}$ increases. But this is obvious from \cite{Z-bal} (Corollaries~4.2 and 4.9), because for any $K,K'\in\mathfrak C_A$ such that $K'\geqslant K$, we have $\delta^K=(\delta^{K'})^K$, whence $\delta^K(\mathbb R^n)\leqslant\delta^{K'}(\mathbb R^n)$.

\subsection{Proof of Corollary~\ref{cortm}}\label{sec-exc} As noted in Theorem~\ref{th3'}, $\lambda_{A,f}\in\Lambda_{A,f}$,  the class $\Lambda_{A,f}$ being introduced in (\ref{gamma}). We thus only need to show that for any given $\mu\in\Lambda_{A,f}$,
 \begin{equation}\label{ml}
 \lambda_{A,f}(\mathbb R^n)\leqslant\mu(\mathbb R^n).
 \end{equation}
But according to Theorem~\ref{th3'}(ii), then necessarily
\[U^{\lambda_{A,f}}\leqslant U^\mu\text{ \ everywhere on $\mathbb R^n$},\]
and (\ref{ml}) follows by use of
the principle of positivity of mass \cite[Theorem~3.11]{FZ}.

\subsection{Proof of Corollary~\ref{cor2}}\label{sec-pr5} Let the assumptions of (a) be fulfilled. Due to the inequality $\delta^A(\mathbb R^n)\leqslant\delta(\mathbb R^n)$ \cite[Corollary~4.9]{Z-bal}, then necessarily
$\delta^A(\mathbb R^n)<1$, which implies, by use of Theorem~\ref{th3'}, that Problem~\ref{pr-main} indeed has no solution.

Let now $\delta\in\breve{\mathcal E}^+(A)$. In view of Theorem~\ref{th3'}, it is enough to consider the case where $c_*(A)=\infty$. The orthogonal projection of $\delta$ onto the (strongly closed by $(\mathcal P_4)$, hence strongly complete) cone $\mathcal E^+(A)=\mathcal E'(A)$ is certainly the same $\delta$, which means, by virtue of \cite[Theorem~3.1(b)]{Z-arx-22}, that
$\delta^A=\delta$,
whence $\delta^A(\mathbb R^n)=1$. Applying Theorem~\ref{th3'} once again, we conclude that the solution $\lambda_{A,f}$ does exist, and moreover $\lambda_{A,f}=\delta^A=\delta$, cf.\ the latter equality in (\ref{RR}). This completes the whole proof.

\subsection{Proof of Theorem~\ref{th-nonth}}\label{sec-pr6} Let $A$ be not inner $\alpha$-thin at infinity; then necessarily $c_*(A)=\infty$.
For (a$_1$), assume $\delta(\mathbb R^n)=1$. Since, by virtue of \cite[Corollary~5.3]{Z-bal2},
\[\delta^A(\mathbb R^n)=\delta(\mathbb R^n)=1,\]
Theorem~\ref{th3'} shows that $\lambda_{A,f}$ does indeed exist, and moreover $\lambda_{A,f}=\delta^A$ (cf.\ the latter equality in (\ref{RR})). Also, $c_{A,f}=0$, which is obvious from (\ref{eqalt}) and (\ref{const-alt}).

For (b$_1$), assume that $f$ is l.s.c.\ on $\overline{A}$. As seen from (a$_1$) and Corollary~\ref{cor2}(a), formula (\ref{iff}) will be proved once we verify the solvability of Problem~\ref{pr-main} in the case
\begin{equation}\label{D}
\delta(\mathbb R^n)>1.
\end{equation}
For the extremal measure $\xi=\xi_{A,f}$, suppose first that $C_\xi\geqslant0$, where $C_\xi$ is the (finite) constant appearing in Lemma~\ref{LLL}. Then, by (\ref{pot-extr}),
\[U^\xi\geqslant U^\delta+C_\xi\geqslant U^\delta\text{ \ n.e.\ on $A$.}\]
The set $A$ not being inner $\alpha$-thin at infinity, an application of the strengthened version of the principle of positivity of mass \cite[Theorem~1.2]{Z-Deny} gives
\[1<\delta(\mathbb R^n)\leqslant\xi(\mathbb R^n),\]
cf.\ (\ref{D}), which however contradicts the inequality $\xi(\mathbb R^n)\leqslant1$ (Remark~\ref{tmass'}). Thus
\begin{equation}\label{Cxi}
C_\xi<0.
\end{equation}
But, by (\ref{pot-extr2'}), $U^\xi_f=C_\xi$ holds true $\xi$-a.e., which implies by integration
\begin{equation*}\label{comb}
\int U^\xi_f\,d\xi=C_\xi\cdot\xi(\mathbb R^n).
\end{equation*}
Since $C_\xi\ne0$,
\[\xi(\mathbb R^n)=\frac{\int U^\xi_f\,d\xi}{C_\xi},\]
whence $\xi(\mathbb R^n)=1$, by (\ref{Xi}). By virtue of Lemma~\ref{LL}, the extremal measure $\xi$ serves, therefore, as the solution $\lambda_{A,f}$. Substituting $\lambda_{A,f}=\xi$ into (\ref{cc}) we finally obtain
\[c_{A,f}=\int U^\xi_f\,d\xi,\]
which combined with (\ref{Xi}) and (\ref{Cxi}) proves (\ref{F1}).

\subsection{Proof of Theorem~\ref{th-th}}\label{sec-pr7} Under the assumptions of the theorem, $c_*(A)=\infty$ and $\delta(\mathbb R^n)\leqslant1$. Therefore, by Theorem~\ref{th3'}, $\lambda_{A,f}$ exists if and only if $\delta^A(\mathbb R^n)=1$.
We aim to show that, due to the remaining requirements of the theorem, we actually have
$\delta^A(\mathbb R^n)\ne1$, and so $\lambda_{A,f}$ fails to exist.
Since $\overline{A}$ is $\alpha$-thin at infinity, there exists the $\alpha$-Riesz equilibrium measure $\gamma$ of $\overline{A}$, treated in an extended sense where $I(\gamma)$ as well as $\gamma(\mathbb R^n)$ may be infinite (for more details, see \cite[Section~V.1.1]{L}, cf.\ \cite[Section~5]{Z-bal} and \cite[Sections~1.3, 2.1]{Z-bal2}). Applying \cite[Theorem~8.7]{Z-bal}, we therefore get
\[\bigl(\delta|_{\Omega_{\overline{A}}}\bigr)^{\overline{A}}(\mathbb R^n)<\delta|_{\Omega_{\overline{A}}}(\mathbb R^n),\]
whence
\[\delta^A(\mathbb R^n)<1,\] for, in consequence of \cite[Corollaries~4.2, 4.9]{Z-bal},
\[\bigl(\delta|_{\Omega_{\overline{A}}}\bigr)^A(\mathbb R^n)=\Bigl(\bigl(\delta|_{\Omega_{\overline{A}}}\bigr)^{\overline{A}}\Bigr)^A(\mathbb R^n)\leqslant\bigl(\delta|_{\Omega_{\overline{A}}}\bigr)^{\overline{A}}(\mathbb R^n).\]

\subsection{Proof of Theorem~\ref{conv6}}\label{sec-pr8'} By Theorems~\ref{th2'} and \ref{th-nonth}, the minimizers $\lambda_{A_j,f}$, $j\in\mathbb N$, and $\lambda_{A,f}$ do exist. (Here we use the fact that for closed sets, $(\mathcal P_4)$ necessarily holds, cf.\ Theorem~\ref{l-quasi}.) Therefore, applying Theorem~\ref{th4'c} we obtain (\ref{cont-count1})--(\ref{cont-count13}).

\subsection{Proof of Theorem~\ref{conv7}}\label{sec-pr8} Under the assumptions of the theorem, $\lambda_{A_j,f}$, $j\in\mathbb N$, and $\lambda_{A,f}$ do exist, see Theorem~\ref{th-nonth}(a$_1$). (Here it should be taken into account that for quasiclosed sets, $(\mathcal P_4)$ necessarily holds, see Theorem~\ref{l-quasi}, and that a countable intersection of quasiclosed sets is likewise quasiclosed, see \cite[Lemma~2.3]{F71}.)
Fix $k\in\mathbb N$. Applying the former relation in (\ref{iff''}) to each of $A_j$ as well as to $A$ gives
\begin{gather}
 \lambda_{A_j,f}=\delta^{A_j}=\bigl(\delta^{\overline{A_k}}\bigr)^{A_j}\text{ \ for all $j\geqslant k$},\label{G1}\\
 \lambda_{A,f}=\delta^A=\bigl(\delta^{\overline{A_k}}\bigr)^{A},\label{G2}
\end{gather}
the latter equality in (\ref{G1}), resp.\ (\ref{G2}), being valid by virtue of \cite[Corollary~4.2]{Z-bal}.
But according to \cite[Theorem~4.10]{Z-arx1} applied to $\delta^{\overline{A_k}}\in\mathcal E^+$, cf.\ (\ref{del-f}),
\begin{gather*}
 \bigl(\delta^{\overline{A_k}}\bigr)^{A_j}\to\bigl(\delta^{\overline{A_k}}\bigr)^{A}\text{ \ strongly and vaguely in $\mathcal E^+$ as $j\to\infty$},\\
 U^{(\delta^{\overline{A_k}})^{A_j}}\downarrow U^{(\delta^{\overline{A_k}})^{A}}\text{ \ pointwise on $\mathbb R^n$ as $j\to\infty$},
\end{gather*}
which combined with (\ref{G1}) and (\ref{G2}) establishes (\ref{cont-count12}) and (\ref{cont-count1du}).

By making use of (\ref{cont-count12}), in the same manner as in the proof of Lemma~\ref{LL} we get
\[\lim_{j\to\infty}\,\int U^\delta\,d\lambda_{A_j,f}=\lim_{j\to\infty}\,\langle\delta^A,\lambda_{A_j,f}\rangle=\langle\delta^A,\lambda_{A,f}\rangle=\int U^\delta\,d\lambda_{A,f},\]
which together with
\[\lim_{j\to\infty}\,\|\lambda_{A_j,f}\|^2=\|\lambda_{A,f}\|^2\]
results in (\ref{cont-count1d}).

To complete the proof, observe that the remaining relation (\ref{cont-count13}) is obvious, because under the stated assumptions, $c_{A_j,f}=c_{A,f}=0$, see the latter relation in (\ref{iff''}).

\subsection{Proof of Theorem~\ref{th-desc}}\label{123}As seen from \cite[Corollary~5.3]{Z-bal2}, in either of the cases (a$_2$) or (b$_2$), (\ref{bal1}) is fulfilled. Therefore, by Theorem~\ref{th3'}, $\lambda_{A,f}$ does exist; it is representable by the former equality in (\ref{RR}) if (a$_2$) occurs, or by the latter  equality otherwise. Applying \cite[Theorems~7.2, 8.5]{Z-bal}, providing  descriptions of the supports of capacitary and swept measures, we obtain (\ref{det5}).
(Regarding the first term on the right-hand side in the latter equality, take into account that, due to $(\mathcal P_2')$ and $(\mathcal P_3')$, we have $\delta|_A\in\mathcal E^+(A)$, whence $(\delta|_A)^A=\delta|_A$, $(\delta|_A)^A$ being the orthogonal projection of $\delta|_A$ onto the strongly closed by $(\mathcal P_4)$, hence strongly complete, convex cone $\mathcal E^+(A)$.)

\subsection{Proof of Theorem~\ref{th-desc2}}\label{sec-pr9} Under the hypotheses of the theorem, $\lambda_{A,f}$ does exist, see Theorem~\ref{th-nonth}(b$_1$). Assuming to the contrary that $S(\lambda_{A,f})$ is noncompact, we conclude by making use of (\ref{t1-pr2})
that there must exist a sequence $(x_j)\subset A$ approaching the Alexandroff point $\infty_{\mathbb R^n}$ as $j\to\infty$, and such that
\[f(x_j)\leqslant U_f^{\lambda_{A,f}}(x_j)=c_{A,f}<0\text{ \ for all $j\in\mathbb N$},\]
the latter inequality being valid by (\ref{F1}). On account of (\ref{limm}), we thus have
\[\lim_{x\to\infty_{\mathbb R^n}, \ x\in A}\,U^\delta(x)>0.\]
On the other hand, it is clear from \cite[Theorem~2.1(ii)]{Z-bal2} that this limit must be equal to $0$, the set $A$ being not inner $\alpha$-thin at infinity. Contradiction.

\section{On the possible extensions of the established theory} If $\omega=0$, where $\omega\in\mathfrak M(\mathbb R^n)$ is the measure appearing in representation (\ref{f1}), then some of the results of this study have already been extended to an arbitrary perfect kernel $\kappa$ on a locally compact Hausdorff space $X$, satisfying the first and the second maximum principles, see \cite[Theorems~1.2, 1.5]{Z-Oh}. However, a further progress in this direction would require the development of a theory of balayage for Radon measures on $X$ of {\it infinite} energy, which is so far an open question, cf.\ \cite[Problem~7.1]{Z-arx}.

\section{Acknowledgements} The author is deeply indebted to D.P.\ Hardin
for reading and commenting on the manuscript. 

\end{document}